\documentclass[10pt,a4paper,english,reqno,article,twosided]{amsart}
\usepackage[T1]{fontenc}
\usepackage[latin9]{inputenc}
\usepackage{amsthm}
\usepackage{amsbsy}
\usepackage{amstext}
\usepackage{amssymb}
\usepackage[all]{xy}

\makeatletter

\numberwithin{equation}{section}
\numberwithin{figure}{section}
\theoremstyle{plain}
\newtheorem{thm}{Theorem}
  \theoremstyle{plain}
  \newtheorem{prop}[thm]{Proposition}
  \theoremstyle{plain}
  \newtheorem{lem}[thm]{Lemma}
  \theoremstyle{remark}
  \newtheorem{rem}[thm]{Remark}
  \theoremstyle{plain}
  \newtheorem{cor}[thm]{Corollary}
  \theoremstyle{remark}
  \newtheorem*{rem*}{Remark}

\newcommand{\Ga}{\Gamma}\newcommand{\Z}{\mathbb{Z}}\newcommand{\italic}{\textit}\newcommand{\stimes}{\boxtimes}\newcommand{\one}{\textbf{1}}\newcommand{\C}{\mathcal{C}}\newcommand{\D}{\mathcal{D}}\newcommand{\N}{\mathcal{N}}\newcommand{\M}{\mathcal{M}}\newcommand{\eL}{\mathcal{L}}\newcommand{\el}{\mathcal{L}}\newcommand{\del}{\partial}\newcommand{\dtimes}{\stimes_\D}\newcommand{\ty}{\mathcal{TY}(A,\chi,\tau)}\newcommand{\mhpsi}{\mathcal{M}(H,\psi)}\newcommand{\nizav}{\perp}\newcommand{\kpsih}{k^{\psi}H}

\makeatother

\usepackage{babel}

\makeatother

\usepackage{babel}

\begin{document}

\title[Module categories]{{\normalsize Module categories over graded fusion categories}}

\author{Ehud Meir and Evgeny Musicantov}
\begin{abstract}
{\normalsize Let $\C$ be a fusion category which is an extension
of a fusion category $\D$ by a finite group $G$. We classify module
categories over $\C$ in terms of module categories over $\D$ and
the extension data $(c,M,\alpha)$ of $\C$. We also describe functor
categories over $\C$ (and in particular the dual categories of $\C$).
We use this in order to classify module categories over the Tambara
Yamagami fusion categories, and their duals.}{\normalsize \par}
\end{abstract}
\maketitle
\global\long\global\long\def\qC{\mathcal{C}}
 \global\long\global\long\def\qD{\mathcal{D}}
 \global\long\global\long\def\mmodule{\mathcal{M}}
 \global\long\global\long\def\mbimodule{\mathcal{M}}
 \global\long\global\long\def\nmodule{\mathcal{N}}
 \global\long\global\long\def\dtimes{\boxtimes_{\qD}}
 \global\long\global\long\def\vecatimes{\boxtimes_{Vec_{A}}}
 \global\long\global\long\def\lmodule{\mathcal{L}}
 \global\long\global\long\def\one{\boldsymbol{1}}
 \global\long\global\long\def\cmdual{\qC_{\mmodule}^{*}}


\section{Introduction}

Let $\C$ be a fusion category. We say that $\C$ is an extension
of the fusion category $\D$ by a finite group $G$ if $\C$ is faithfully
graded by the group $G$ in such a way that $\C_{e}=\D$. In \cite{ENO2}
Etingof et. al. classified extension of a given fusion category $\D$
by a given finite group $G$. Their classification is given by a triple
$(c,M,\alpha)$, where $c:G\rightarrow Pic(\D)$ is a homomorphism,
$M$ belongs to a torsor over $H^{2}(G,inv(Z(\D)))$, and $\alpha$
belongs to a torsor over $H^{3}(G,k^{*})$. The group $Pic(\D)$ is
the group of invertible $\D$-bimodules (up to equivalence), and the
group $inv(Z(\D))$ is the group of (isomorphism classes of) invertible
objects in the center $Z(\D)$ of $\D$.

Let us recall briefly the construction from \cite{ENO2}. Suppose
that we are given a classification data $(c,M,\alpha)$. The corresponding
category $\C$ will be $\bigoplus_{g\in G}c(g)$ as a $\D$-bimodule
category. If we choose arbitrary isomorphisms $c(g)\stimes_{\D}c(h)\rightarrow c(gh)$
for the tensor product in $\C$, the multiplication will not necessarily
be associative. This non associativity is encoded in a cohomological
obstruction $O_{3}(c)\in Z^{3}(G,inv(Z(\D)))$. The element $M$ belongs
to $C^{2}(G,inv(Z(\D)))$, and should satisfy $\del M=O_{3}(c)$ (that
is- it should be a {}``solution'' to the obstruction $O_{3}(c)$).
If we change $M$ by a coboundary, we get an equivalent solution.
Therefore, the choice of $M$ is equivalent to choosing an element
from a torsor over $H^{2}(G,inv(Z(\D)))$. Given $c$ and $M$, we
still have one more obstruction in order to furnish from $\C$ a fusion
category. This obstruction is the commutativity of the pentagon diagram,
and is given by a four cocycle $O_{4}(c,M)\in Z^{4}(G,k^{*})$. The
element $\alpha$ belongs to $C^{3}(G,k^{*})$, and should satisfy
$\del\alpha=O_{4}(c,M)$. We think of $\alpha$ as a solution to the
obstruction $O_{4}(c,M)$. Again, if we change $\alpha$ by a coboundary,
we will get an equivalent solution. Therefore, the choice of $\alpha$
can be seen as a choice from a torsor over $H^{3}(G,k^{*})$.

We shall write $\C=\D(G,c,M,\alpha)$ to indicate the fact that $\C$
is an extension of $\D$ by $G$ given by the extension data $(c,M,\alpha)$,
and we shall assume from now on that $\C=\D(G,c,M,\alpha)$.

In this paper we shall classify module categories over $\C$ in terms
of module categories over $\D$ and the extension data $(\C,M,\alpha)$.

Our classification of module categories will follow the lines of the
classification of \cite{ENO2}. We will begin by proving the following
structure theorem for module categories over $\C$.
\begin{thm}
\label{pro:structure}Let $\el$ be an indecomposable module category
over $\C$. There is a subgroup $H<G$, and an indecomposable $\C_{H}=\bigoplus_{a\in H}\C_{a}$
module category $\N$ which remains indecomposable over $\D$ such
that $\lmodule\cong Ind_{\qC_{H}}^{\qC}(\nmodule)\triangleq\qC\boxtimes_{\qC_{H}}\nmodule$.
\end{thm}
This proposition enables us to reduce the classification of $\C$-module
categories to the classification of $\C_{H}$-module categories which
remains indecomposable over $\D$, where $H$ varies over subgroups
of $G$.

In order to classify such categories we will go, in some sense, the
other way around. We will begin with an indecomposable $\D$-module
category $\N$, and we will ask how can we equip $\N$ with a structure
of a $\C_{H}$ module category.

As in the classification in \cite{ENO2}, the answer will also be
based upon choosing solutions to certain obstruction (in case it is
possible). We will begin with the observation, in Section \ref{sec:zero},
that we have a natural action of $G$ on the set of (equivalence classes
of) indecomposable $\qD$-module categories. This action is given
by the following formula \[
g\cdot\N=\C_{g}\stimes_{\D}\N.\]
 If $\N$ has a structure of a $\C_{H}$-module category, then the
action of $\C_{H}$ on $\N$ will give an equivalence of $\qD$-module
categories $h\cdot\nmodule\cong\nmodule$ for every $h\in H$. In
other words- $\N$ will be $H$-invariant. We may think of the fact
that $\N$ should be $H$-invariant as the {}``zeroth obstruction''
we have in order to equip $\N$ with a structure of a $\C_{H}$-module
category.

In case $\N$ is $H$-invariant, we choose equivalences $\psi_{a}:\C_{a}\stimes_{\D}\N\rightarrow\N$
for every $a\in H$. We would like these equivalences to give us a
structure of a $\C_{H}$-module category on $\N$. As one might expect,
not every choice of equivalences will do that. If $\N$ has a structure
of a $\C_{H}$-module category, we will see in Section \ref{sec:The-first-two}
that we have a natural action of $H$ on the group $\Gamma=Aut_{\D}(\N)$.
In case we only know that $\N$ is $H$-invariant, we only have an
\textit{outer} action of $H$ on $\Gamma$ (i.e. a homomorphism $\rho:H\rightarrow Out(\Gamma)$).
The first obstruction will thus be the possibility to lift this outer
action to a proper action.

Once we overcome this obstruction (and choose a lifting $\Phi$ for
the outer action), our second obstruction will be the fact that the
two functors \[
F_{1},F_{2}:\C_{a}\stimes_{\D}\C_{b}\stimes_{\D}\stimes\N\rightarrow\N\]
 defined by \[
F_{1}(X\stimes Y\stimes N)=(X\otimes Y)\otimes N\]
 and \[
F_{2}(X\stimes Y\stimes N)=X\otimes(Y\otimes N)\]
 should be isomorphic. We will see that this obstruction is given
by a certain two cocycle $O_{2}(\N,c,H,M,\Phi)\in Z^{2}(H,Z(Aut_{\D}(\N)))$.
A solution for this obstruction is an element $v\in C^{1}(H,Z(Aut_{\D}(\N)))$
that should satisfy $\del v=O_{2}(\N,c,H,M,\Phi)$.

Our last obstruction will be the fact that the above functors should
be not only isomorphic, but they should be isomorphic in a way which
will make the pentagon diagram commutative. This obstruction is encoded
by a three cocycle $O_{3}(\N,c,H,M,\Phi,v,\alpha)\in Z^{3}(H,k^{*})$.
A solution $\beta$ for this obstruction will be an element of $C^{2}(H,k^{*})$
such that $\del\beta=O_{3}(\N,c,H,M,\Phi,v,\alpha)$.

We can summarize our main result in the following theorem:
\begin{thm}
\label{thm:main} An indecomposable module category over $\C$ is
given by a tuple $(\N,H,\Phi,v,\beta)$, where $\N$ is an indecomposable
module category over $\D$, $H$ is a subgroup of $G$ which acts
trivially on $\N$, $\Phi:H\rightarrow Aut(Aut_{\D}(\N))$ is a homomorphism,
$v$ belongs to a torsor over $H^{1}(H,Z(Aut_{\D}(\N)))$, and $\beta$
belongs to a torsor over $H^{2}(H,k^{*})$.
\end{thm}
We shall denote the indecomposable module category which corresponds
to the tuple $(\N,H,\Phi,v,\beta)$ by $\M(\N,H,\Phi,v,\beta)$. In
order to classify module categories, we need to give not only a list
of all indecomposable module categories, but also to explain when
does two elements in the list define equivalent module categories.
We will see in Section \ref{sec:The-isomorphism-condition} that if
$\M(\N,H,\Phi,v,\beta)$ is any indecomposable module category, $g\in G$
is an arbitrary element and $F:\C_{g}\stimes_{\D}\N\equiv\N'$ is
an equivalence of $\D$-module categories (where $\N'$ is another
indecomposable $\D$-module category), then $F$ gives rise to a tuple
$(\N',gHg^{-1},\Phi',v',\beta')$ which defines an equivalent $\C$-module
category. Our second main result is the following:
\begin{thm}
\label{thm:2} Two tuples $(\N,H,\Phi,v,\beta)$ and $(\N',H',\Phi',v',\beta')$
determine equivalent $\C$-module categories if and only if the second
tuple is defined by the first tuple and by some equivalence $F$ as
above.
\end{thm}
We shall prove Theorem \ref{thm:2} in Section \ref{sec:The-isomorphism-condition}.
We will also decompose this condition into a few simpler ones: we
will see, for example, by considering the case $g=1$, that we can
change $\Phi$ to be $t\Phi t^{-1}$, where $t$ is any conjugation
automorphism of $Aut_{\D}(\N)$.

In Section \ref{sec:Equivaraintization} we will describe the category
of functors $Fun_{\C}(\N,\M)$ where $\N$ and $\M$ are two module
categories over $\C$. We will prove a Mackey type decomposition theorem, and we will also see that we can view this category
as the equivariantization of the category $Fun_{\D}(\N,\M)$ with
respect to an action of $G$. We will also be able to prove the following
criterion of $\C$ to be group theoretical: $\qC$ is a group theoretical
if and only if there is a pointed $\qD$-module category $\nmodule$
(i.e., $\qD_{\nmodule}^{*}$ is pointed), stable under the $G$-action,
i.e., for every $g\in G$, $\qC_{g}\dtimes\nmodule\cong\nmodule$
as $\qD$-module categories.
We shall also explain why this is a reformulation of the criterion which appears in \cite{GNN}.

A theorem of Ostrik says that any indecomposable module category over
a fusion category $\D$ is equivalent to a category of the form $Mod_{\D}-A$,
of right $A$-modules in the category $\D$, where $A$ is some semisimple
indecomposable algebra in the category $\D$. In other words- any
module category has a description by objects which lie inside the
fusion category $\D$. In Section \ref{sec:An-intrinsic-description}
we will explain how we can understand the obstructions and their solutions,
and also the functor categories, by intrinsic description; that is-
by considering algebras and modules inside the categories $\D$ and
$\C$.

This description will be much more convenient for calculations. It
will also enables us to view the first and the second obstruction
in a unified way. Indeed, in Section \ref{sec:An-intrinsic-description}
we will show that we have a natural short exact sequence \[
1\rightarrow\Gamma\rightarrow\Lambda\rightarrow H\rightarrow1\]
 and that a solution for the first two obstructions is equivalent
to a choice of a splitting of this sequence (and therefore, we can
solve the first two obstructions if and only if this sequence splits).
We will also show, following the results of Section \ref{sec:An-intrinsic-description},
that two splittings which differ by conjugation by an element of $\Gamma$
will give us equivalent module categories.

In Section \ref{sec:An-example:-classification} we shall give a detailed
example. We will consider the Tambara Yamagami fusion categories,
$\C=\mathcal{TY}(A,\chi,\tau)$. In this case $\C$ is an extension
of the category $Vec_{A}$, where $A$ is an abelian group, by the
group $\Z_{2}$.

\begin{rem*} During the final stages of the writing of this paper it came to our attention that Cesar Galindo is working on a paper with similar results.
We would like to remark that our results and his were obtained independently.
\end{rem*}

\section{Preliminaries\label{sec:Preliminaries}}

In this section, $\C$ will be a general fusion category and $\D$
a fusion subcategory of $\C$. We recall some basic facts about module
categories over $\C$ and $\D$. For a more detailed discussion on
these notions, we refer the reader to \cite{O1} and to \cite{ENO}.
Let $\N$ be a module category over $\C$. If $X,Y\in Ob\N$, then
the \textit{internal hom} of $X$ and $Y$ is the unique object of
$\C$ which satisfies the formula \[
Hom_{\C}(W,\underline{Hom}_{\C}(X,Y))=Hom_{\N}(W\otimes X,Y)\]
 for every $W\in Ob\C$. For every $X\in Ob\N$ the object $\underline{Hom}_{\C}(X,X)$
has a canonical algebra structure. We say that $X$ \textit{generates}
$\N$ (over $\C$) if $\N$ is the smallest sub $\C$-module-category
of $\N$ which contains $X$. For every algebra $A$ in $\C$, $mod_{\C}-A$,
the category of right $A$-modules in $\C$, has a structure of a
left $\C$-module category.

A theorem of Ostrik says that all module categories are of this form:
\begin{thm}
(see \cite{O1}) Let $\N$ be a module category, and let $X$ be a
generator of $\N$ over $\C$. We have an equivalence of $\C$-module
categories $\N\cong Mod_{\C}-\underline{Hom}(X,X)$ given by $F(Y)=\underline{Hom}(X,Y)$.
\end{thm}
Next, we recall the definition of the induced module category. If
$\N$ is a $\D$-module category, $Ind_{\D}^{\C}(\N)$ is a module
category over $\C$ which satisfies Frobenius reciprocity. This means
that for every $\C$-module category $\mathcal{R}$ we have that \[
Fun_{\C}(Ind_{\D}^{\C}(\N),\mathcal{R})\cong Fun_{\D}(\N,\mathcal{R}).\]

The next lemma proves that the induced module category always exists.
It will also gives us some idea about how the induced module category
{}``looks like''.
\begin{lem}
Suppose that $\N\cong mod_{\D}-A$ for some algebra $A\in Ob\D$.
Then $A$ can also be considered as an algebra in $\C$, and $Ind_{\D}^{\C}(\N)\cong mod_{\C}-A$.\end{lem}
\begin{proof}
Let us prove that Frobenius reciprocity holds. For this, we first
need to represent $\mathcal{R}$ in an appropriate way. We choose
a generator $X$ of $\mathcal{R}$ \textbf{over $\D$}. It is easy
to see that $X$ is also a generator over $\C$. Then, by Ostrik's
Theorem we have that $\mathcal{R}\cong mod_{\C}-\underline{Hom}_{\C}(X,X)$
over $\C$, and $\mathcal{R}\cong mod_{\D}-\underline{Hom}_{\D}(X,X)$
over $\D$. If we denote $\underline{Hom}_{\C}(X,X)$ by $B$, then
it is easy to see by the definition of $\underline{Hom}$ that $\underline{Hom}_{\D}(X,X)\cong B_{\D}$,
where $B_{\D}$ is the largest subobject of $B$ which is also an
object of $\D$ (since $\D$ is a fusion subcategory of $\C$, this
is also a subalgebra of $B$). By another theorem of Ostrik (see \cite{O1}),
we know that $Fun_{\C}(mod_{\C}-A,mod_{\C}-B)\cong bimod_{\C}-A-B$.
Using the theorem of Ostrik again, we see that $Fun_{\D}(\N,\mathcal{R})\cong bimod_{\D}(A-B_{\D})$.
One can verify that the functor which sends an $A-B_{\D}$ bimodule
$Z$ in $\D$ to $Z\otimes_{B_{\D}}B$ gives an equivalence between
the two categories.\end{proof}
\begin{rem}
The fact that the induction functor is an equivalence of categories
arise from the fact that for such a $B$, the equivalence between
the categories $mod_{\D}-B_{\D}$ and $mod_{\C}-B$ is given by $X\mapsto X\otimes_{B_{\D}}B$.

One can show that the induced module category is also equivalent to
$\C\stimes_{\D}\N$.
\end{rem}
In particular, we have the following:
\begin{cor}
\label{cor1} Let $\qC$ be a fusion category and let $\qD$ be a
fusion subcategory of $\qC$. Let $\N$ be a module category over
$\C$. Suppose that $X$ is a generator of $\N$ over $\C$, and that
the algebra $A=\underline{Hom}(X,X)$ is supported on $\D$. Then
$\N\cong Ind_{\D}^{\C}(mod_{\D}-A)$.
\end{cor}

\section{Decomposition of the module category over the trivial component subcategory.
The zeroth obstruction\label{sec:zero}}

We begin by considering the action of $G$ on $\D$-module categories.
For every $g\in G$, $\C_{g}$ is an invertible $\D$-bimodule category.
Therefore, if $\N$ is an indecomposable $\D$-module category, the
category $\C_{g}\stimes_{\D}\N$ is also indecomposable. It is easy
to see that we get in this way an action of $G$ on the set of (equivalence
classes of) indecomposable $\D$-module categories. Let now $\eL$
be an indecomposable $\C$-module category. We can consider $\eL$
also as a module category over $\D$. We claim the following:
\begin{lem}
As a $\D$-module category, $\eL$ is $G$-invariant.\end{lem}
\begin{rem}
For this lemma, we do not need to assume that $\eL$ is indecomposable.\end{rem}
\begin{proof}
We have the following equivalences of $\D$-module categories \[
\C_{g}\stimes_{\D}\eL\cong\C_{g}\stimes_{\D}(\C\stimes_{\C}\eL)\cong\]
 \[
(\C_{g}\stimes_{\D}\C)\stimes_{\C}\eL\cong(\C_{g}\stimes_{\D}\oplus_{a\in G}\C_{a})\stimes_{\C}\el\cong\]
 \[
(\oplus_{a\in G}\C_{ga})\stimes_{\C}\el\cong\C\stimes_{\C}\el\cong\el.\]
 This proves the claim.
\end{proof}
If $H$ is a subgroup of $G$, we have the subcategory $\C_{H}=\bigoplus_{h\in H}C_{h}$
of $\C$, which is an extension of $\D$ by $H$. We claim the following:
\begin{prop}
There is a subgroup $H<G$, and an indecomposable $\C_{H}$ module
category $\N$ which remains indecomposable over $\D$ such that $\el\equiv Ind_{\C_{H}}^{\C}(\N)$. \end{prop}
\begin{proof}
Suppose that $\eL$ decomposes over $\qD$ as \[
\eL=\bigoplus_{i=1}^{n}\eL_{i}.\]
 For every $g\in G$, we have seen that the action functor defines
an equivalence of categories $\C_{g}\stimes_{\D}\eL\cong\eL$. Since
\[
\C_{g}\stimes_{\D}\eL\cong\bigoplus_{i=1}^{n}\C_{g}\stimes_{\D}\eL_{i},\]
 we see that $G$ permutes the index set $\{1,\ldots,n\}$. This action
is transitive, as otherwise $\eL$ would not have been indecomposable
over $\C$. Let $H<G$ be the stabilizer of $L_{1}$. Then $\N=\eL_{1}$
is a $\C_{H}$-module category which remains indecomposable over $\D$.
Let $X\in Ob\eL_{1}$ be a generator of $\eL$ over $\C$ (any nonzero
object would be a generator, as $\eL$ is indecomposable over $\C$).
By the fact that the stabilizer of $\eL_{1}$ is $H$, it is easy
to see that $\underline{Hom}_{\C}(X,X)$ is contained in $\C_{H}$.
The rest of the lemma now follows from corollary \ref{cor1}.
\end{proof}
So in order to classify indecomposable module categories over $\C$,
we need to classify, for every $H<G$, the indecomposable module categories
over $\C_{H}$ which remain indecomposable over $\D$. For every indecomposable
module category $\eL$ over $\C$, we have attached a subgroup $H$
of $G$ and an indecomposable $\C_{H}$ module category $\eL_{1}$
which remains indecomposable over $\D$. The subgroup $H$ and the
module category $\eL_{1}$ will be the first two components of the
tuple which corresponds to $\eL$. Notice that we could have chosen
any conjugate of $H$ as well.

\section{The first two obstructions\label{sec:The-first-two}}

Let $\eL$, $\N=\eL_{1}$ and $H$ be as in the previous section.
For every $a\in H$ we have an equivalence of $\D$-module categories
$\psi_{a}:\C_{a}\stimes_{\D}\N\cong N$ given by the action of $\C_{H}$
on $\N$. Suppose on the other hand that we are given an $H$-invariant
indecomposable module category $\N$ over $\D$. Let us fix a family
of equivalences $\{\psi_{a}\}_{a\in H}$, where $\psi_{a}:\C_{a}\stimes\N\rightarrow\N$.
Let us see when does this family comes from an action of $\C_{H}$
on $\N$.

We know that the two functors \[
\C_{H}\stimes\C_{H}\stimes\N\stackrel{m\stimes1_{\N}}{\rightarrow}\C_{H}\stimes\N\stackrel{\cdot}{\rightarrow}\N\]
 and \[
\C_{H}\stimes\C_{H}\stimes\N\stackrel{1_{\C_{H}}\stimes(\cdot)}{\rightarrow}\C_{H}\stimes\N\stackrel{\cdot}{\rightarrow}\N\]
 should be isomorphic.

Since the action of $\C_{H}$ on $\N$ is given by the action of $\D$
together with the $\psi_{a}$'s, this condition translates to the
fact that for every $a,b\in H$ the two functors \[
\C_{a}\stimes_{\D}\C_{b}\stimes_{\D}\N\stackrel{M_{a,b}\stimes1_{\N}}{\rightarrow}\C_{ab}\stimes_{\D}\N\stackrel{\psi_{ab}}{\rightarrow}\N\]
 and \[
\C_{a}\stimes_{\D}\C_{b}\stimes_{\D}\N\stackrel{1_{\C_{a}}\stimes\psi_{b}}{\rightarrow}\C_{a}\stimes_{\D}\N\stackrel{\psi_{a}}{\rightarrow}\N\]
 should be isomorphic. We can express this condition in the following
equivalent way- for every $a,b\in H$, the autoequivalence of $\N$
as a $\D$-module category \[
Y_{a,b}=\N\stackrel{\psi_{a}^{-1}}{\rightarrow}\C_{a}\stimes_{\D}\N\stackrel{1_{\C_{a}}\stimes\psi_{b}^{-1}}{\rightarrow}\C_{a}\stimes_{\D}\C_{b}\stimes_{\D}\N\]
 \[
\stackrel{M_{a,b}\stimes1_{\N}}{\rightarrow}C_{ab}\stimes_{\D}\N\stackrel{\psi_{ab}^{-1}}{\rightarrow}\N\]
 should be isomorphic to the identity autoequivalence. We shall decompose
this condition into two simpler ones.

Consider the group $\Ga=Aut_{\D}(\N)$, where by $Aut_{\D}$ we mean
the group of $\D$-autoequivalences (up to isomorphism) of $\N$.
For $a\in H$ and $F\in\Ga$ define $a\cdot F\in\Ga$ as the composition
\[
\N\stackrel{\psi_{a}^{-1}}{\rightarrow}\C_{a}\stimes_{\D}\N\stackrel{1_{\C_{a}}\stimes F}{\rightarrow}\C_{a}\stimes_{\D}\N\stackrel{\psi_{a}}{\rightarrow}\N.\]
 We get a map $\Phi:H\rightarrow Aut(\Ga)$ given by $\Phi(g)(F)=h\cdot F$.
This map depends on the choice of the $\psi_{a}$'s and is not necessary
a group homomorphism. However, the following equation does hold for
every $a,b\in H$: \begin{equation}
\Phi(a)\Phi(b)=\Phi(ab)C_{Y_{a,b}},\label{cocycle1}\end{equation}
 where we write $C_{x}$ for conjugation by $x\in\Ga$.

Notice that $\psi_{a}$ is determined up to composition with an element
in $\Ga$, and that by changing $\psi_{a}$ to be $\psi_{a}'=\gamma\psi_{a}$,
for $\gamma\in\Ga$, we change $\Phi(a)$ to be $\Phi(a)c_{\gamma},$
where by $c_{\gamma}$ we mean conjugation by $\gamma$. Equation
\ref{cocycle1} shows that the composition $\rho=\pi\Phi$, where
$\pi$ is the quotient map $\pi:Aut(\Ga)\rightarrow Out(\Ga)$ does
give a group homomorphism. Notice that by the observation above, $\rho$
does not depend on the choice of the $\psi_{a}$'s, but only on $c$,
$\N$ and $H$. We have the following
\begin{lem}
Let $\nmodule$ be $H$-invariant $\qD$-module category. There is
a well defined group homomorphism $\rho:H\rightarrow Out(\Gamma)$.
If the $\psi_{a}$'s arise from an action of $\C_{H}$ on $\N$, then
the map $\Phi$ described above is a group homomorphism.\end{lem}
\begin{proof}
This follows from the fact that by the discussion above, if the $\psi_{a}$'s
arise from an action of $\C_{H}$ on $\N$, then $Y_{a,b}$ is trivial
for every $a,b\in H$, and by Equation \ref{cocycle1} we see that
$\Phi$ is a group homomorphism.
\end{proof}
So $c,\N$ and $H$ determines a homomorphism $\rho:H\rightarrow Out(\Ga)$.
We thus see that in order to give $\N$ a structure of a $\C_{H}$-module
category, we need to give a lifting of $\rho$ to a homomorphism to
$Aut(\Ga)$. The first obstruction is thus the possibility to lift
$\rho$ in such a way.

Suppose then that we have a lifting, that is- a homomorphism $\Phi:H\rightarrow Aut(\Ga)$
such that $\pi\Phi=\rho$. To say that $\Phi$ is a homomorphism is
equivalent to say that we have chosen the $\psi_{a}$'s in such a
way that $C_{Y_{a,b}}=Id$, or in other words- in such a way that
for every $a,b\in H$, $Y_{a,b}$ is in $Z(\Ga)$, the center of $\Ga$.

Notice that after choosing $\Phi$, we still have some liberty in
changing the $\psi_{a}$'s. Indeed, if we choose $\psi_{a}'=\gamma_{a}\psi_{a}$,
where $\gamma_{a}\in Z(\Ga)$ for every $a\in H$, we still get the
same $\Phi$, and it is easy to see that every $\psi_{a}'$ that will
give us the same $\Phi$ is of this form.

In order to furnish a structure of a $\C_{H}$-module category on
$\N$, we need $Y_{a,b}$ to be not only central, but trivial. A straightforward
calculation shows now that the function $H\times H\rightarrow Z(\Ga)$
given by $(a,b)\mapsto Y_{a,b}$ is a two cocycle. If we choose a
different set of isomorphisms $\psi_{a}'=\gamma_{a}\psi_{a}$ where
$\gamma_{a}\in Z(\Ga)$, we will get a cocycle $Y'$ which is cohomologous
to $Y$. So the second obstruction is the cohomology class of the
two cocycle $(a,b)\mapsto Y_{a,b}$. We shall denote this obstruction
by $O_{2}(\N,c,H,M,\Phi)\in Z^{2}(H,Z(\Ga))$. Notice that this obstruction
depends linearly on $M$ in the following sense: we have a natural
homomorphism of groups $\xi:inv(Z(\D))\rightarrow\Ga$, given by the
formula \[
\xi(T)(N)=T\otimes N\]
 (that is- $\xi(T)$ is just the autoequivalence of acting by $T$)
It can be seen that if we would have chosen $M'=M\zeta$, where $\zeta\in Z^{2}(G,Z(\D))$,
then we would have changed $O_{2}$ to be $O_{2}res_{H}^{G}(\xi_{*}(\zeta))$.

In conclusion- we saw that if $\N$ is a $\D$-module category upon
which $H$ acts trivially, then we have an induced homomorphism $\rho:H\rightarrow Out(\Ga)$.
The first obstruction to define on $\N$ a structure of a $\C_{H}$-module
category is the fact that $\rho$ should be of the form $\pi\Phi$
where $\Phi:H\rightarrow Aut(Aut_{\D}(\N))$ is a homomorphism. After
choosing such a lifting $\Phi$ we get the second obstruction, which
is a two cocycle $O_{2}(\N,c,H,M,\Phi)\in Z^{2}(H,Z(\Ga))$. A solution
to this obstruction will be an element $v\in C^{1}(H,Z(\Ga))$ which
satisfies \[
\del v=O_{2}(\N,c,H,M,\Phi).\]
 We will see later, in Section \ref{sec:An-intrinsic-description},
that to find a solution for the first and for the second obstruction
is the same thing as to find a splitting for a certain short exact
sequence. We will also see why two solutions $v$ and $v'$ which
differs by a coboundary give equivalent module categories (and therefore
we can view the set of possible solutions, in case it is not empty,
as a torsor over $H^{1}(H,Z(\Ga))$.

\section{The third obstruction}

So far we have almost defined a $\C_{H}$-action on $\N$, by means
of the equivalences $\psi_{a}:\C_{a}\stimes_{\D}\N\rightarrow\N$.
The solutions for the first and for the second obstruction ensures
us that for every $a,b\in H$ the two functors \[
F_{1}:\C_{a}\stimes_{\D}\C_{b}\stimes_{\D}\N\stackrel{M_{a,b}\stimes1_{\N}}{\rightarrow}\C_{ab}\stimes_{\D}\N\stackrel{\psi_{ab}}{\rightarrow}\N\]
 and \[
F_{2}:\C_{a}\stimes_{\D}\C_{b}\stimes_{\D}\N\stackrel{1_{\C_{a}}\stimes\psi_{b}}{\rightarrow}\C_{a}\stimes_{\D}\N\stackrel{\psi_{a}}{\rightarrow}\N\]
 are isomorphic.

For every $a,b\in H$, let us fix an isomorphism $\eta(a,b):F_{1}\rightarrow F_{2}$
between the two functors. In other words, for every $X\in\C_{a}$,
$Y\in\C_{b}$ and $N\in\N$ we have a natural isomorphism \[
\eta(a,b)_{X,Y,N}:(X\otimes Y)\otimes N\rightarrow X\otimes(Y\otimes N).\]
 Since $F_{1}$ and $F_{2}$ are simple as objects in the relevant
functor category (they are equivalences), the choice of the isomorphism
$\eta(a,b)$ is unique up to a scalar, for every $a,b\in H$.

The final condition for $\N$ to be a $\C_{H}$-module category is
the commutativity of the pentagonal diagram. In other words, for every
$a,b,d\in H$, and every $X\in\C_{a}$, $Y\in\C_{b}$, $Z\in\C_{d}$
and $N\in\N$, the following diagram should commute: \[
\xymatrix{(X\otimes(Y\otimes Z))\otimes N\ar[rr]^{\eta(a,bd)_{X,Y\otimes Z,N}} &  & X\otimes((Y\otimes Z)\otimes N)\ar[d]^{\eta(b,d)_{Y,Z,N}}\\
((X\otimes Y)\otimes Z)\otimes N\ar[u]^{\alpha_{X,Y,Z}}\ar[rd]^{\eta(ab,d)_{X\otimes Y,Z,N}} &  & X\otimes(Y\otimes(Z\otimes N))\\
 & (X\otimes Y)\otimes(Z\otimes N)\ar[ru]^{\eta(a,b)_{X,Y,Z\otimes N}}}
\]

This diagram will always be commutative up to a scalar $O_{3}(a,b,d)$
which depends only on $a,b$ and $d$, and not on the particular objects
$X,Y,Z$ and $N$. One can also see that the function $(a,b,d)\mapsto O_{3}(a,b,d)$
is a three cocycle on $H$ with values in $k^{*}$, and that choosing
different $\eta(a,b)$'s will change $O_{3}$ by a coboundary. We
call $O_{3}=O_{3}(\N,c,H,M,\Phi,v,\alpha)\in H^{3}(H,k^{*})$ the
third obstruction. A solution to this obstruction is equivalent to
giving a set of $\eta(a,b)$'s such that the pentagon diagram will
be commutative. We will see in the next section that by altering $\eta$
by a coboundary we will get equivalent module categories. Thus, we
see that the set of solutions for this obstruction will be a torsor
over the group $H^{2}(H,k^{*})$ (in case a solution exists). Notice
that this obstruction depends {}``linearly'' on $\alpha$, in the
sense that if we would have change $\alpha$ to be $\alpha\zeta$
where $\zeta\in H^{3}(G,k^{*})$, then we would have changed the obstruction
by $\zeta$. In other words: \[
O_{3}(\N,c,H,M,\Phi,v,\alpha\zeta)=O_{3}(\N,c,H,M,\Phi,v,\alpha)res_{H}^{G}(\zeta).\]
 This ends the proof of Theorem \ref{thm:main}.

\section{The isomorphism condition\label{sec:The-isomorphism-condition}}

In this section we answer the question of when does the $\C$-module
categories $\M(\N,H,\Phi,v,\beta)$ and $\M(\N',H',\Phi',v',\beta')$
are equivalent.

Assume then that we have an equivalence of $\C$-module categories
\[
F:\M(\N,H,\Phi,v,\beta)\rightarrow\M(\N',H',\Phi',v',\beta').\]
 Let us denote these categories by $\M$ and $\M'$ respectively.
Then $F$ is also an equivalence of $\D$-module categories. Recall
that as $\D$-module categories, $\M$ splits as \[
\bigoplus_{g\in G/H}\C_{g}\stimes_{\D}\N.\]
 A similar decomposition holds for $\M'$.

By considering these decompositions, it is easy to see that $F$ induces
an equivalence of $\D$-module categories between $\C_{g}\stimes_{\D}\N$
and $\N'$ for some $g\in G$. Let us denote the restriction of $F$
to $\C_{g}\stimes_{\D}\N$ as a functor of $\D$-module categories
by $t_{F}$. We can reconstruct the tuple $(\N',H',\Phi',v',\beta')$
from $t_{F}$ in the following way: We have already seen that $\N'$
is equivalent to $\C_{g}\stimes_{\D}\N$ and that the stabilizer subgroup
of the category $\nmodule'$ will be $H'=gHg^{-1}$.

Let us denote by $\Ga'$ the group $Aut_{\D}(N')$. We have a natural
isomorphism $\nu:\Ga\rightarrow\Ga'$ given by the formula \[
\nu(t):\N'\stackrel{F^{-1}}{\rightarrow}\C_{g}\stimes_{\D}\N\stackrel{1\stimes t}{\rightarrow}\C_{g}\stimes_{\D}\N\stackrel{F}{\rightarrow}\N'.\]
 Using the functor $t_{F}$ and the map $\nu$ we can see that the
map \[
\rho':gHg^{-1}\rightarrow Out(\Ga')\]
 which appears in the construction of the second module category is
the composition \[
gHg^{-1}\stackrel{c_{g}}{\rightarrow}H\stackrel{\rho}{\rightarrow}Out(\Ga)\rightarrow Out(\Ga'),\]
 where the last morphism is induced by $\nu$. The map $\Phi'$ which
lifts $\rho'$ will depend on $\Phi$ in a similar fashion. The same
holds for the second obstruction and its solution.

For the third obstruction, the situation is a bit more delicate. Since
$F$ is a functor of $\C$-module categories, we have, for each $a\in H$,
a natural isomorphism between the functors \[
\C_{gag^{-1}}\stimes_{\D}(\C_{g}\stimes_{\D}\N)\stackrel{1\stimes F}{\rightarrow}\C_{gag^{-1}}\stimes_{\D}\N'\stackrel{\cdot}{\rightarrow}\N'\]
 and \[
\C_{gag^{-1}}\stimes_{\D}(\C_{g}\stimes_{\D}\N)\stackrel{\cdot}{\rightarrow}\C_{g}\stimes_{\D}\N\stackrel{F}{\rightarrow}\N'\]

For any $a\in H$, the choice of the natural isomorphism is unique
up to a scalar. A direct calculation shows that if we change the natural
isomorphisms by a set of scalars $\zeta_{a}$, we will get an equivalence
$\M(N,H,\Phi,v,\beta)\rightarrow\M(N',H',\Phi',v',\beta'')$ where
$\beta''=\beta'\del\zeta$. This is the reason that cohomologous solutions
for the third obstruction will give us equivalent module categories.

In conclusion, we have the following:
\begin{prop}
Assume that we have an isomorphism $F:\M(\N,H,\Phi,v,\beta)\rightarrow\M(\N',H',\Phi',v',\beta')$
Then there is a $g\in G$ such that $F$ will induce an equivalence
of $\D$-module categories $\C_{g}\stimes_{\D}\N\rightarrow\N'$,
and the data $(\N',H',\Phi',v',\beta')$ can be reconstructed from
$t_{F}$ in the way described above ($\beta'$ will be reconstructible
only up to a coboundary) .
\end{prop}
Notice that we do not have any restriction on $t_{F}$. In other words,
given any $t_{F}:\C_{g}\stimes_{\D}\N\rightarrow\N'$ we can always
reconstruct the tuple $(\N',H',\Phi',v',\beta')$ in the way described
above.

We would like now to {}``decompose'' the equivalence in the theorem
into several steps. The first ingredient that we need in order to
get an equivalence is an element $g\in G$ such that $\C_{g}\stimes_{\D}\N\equiv\N'$.

Consider now the case where this ingredient is trivial, that is- $g=1$,
$\N=\N'$ and $H=H'$. In that case $t_{F}$ is an autoequivalence
of the $\D$-module category $\N$. Let us denote by $\psi_{a}:\C_{a}\stimes_{\D}\N\rightarrow\N$
and by $\psi_{a}':\C_{a}\stimes_{\D}\N\rightarrow\N$ the structural
equivalences of the two categories (where $a\in H$). Since $F$ is
an equivalence of $\C$-module categories, we see that the following
diagram is commutative:

\[
\xymatrix{\C_{a}\stimes_{\D}\N\ar[r]^{\,\,\,\,\,\,\,\,\,\psi_{a}}\ar[d]^{1\stimes t_{F}} & \N\ar[d]^{t_{F}}\\
\C_{a}\stimes_{\D}\N\ar[r]^{\,\,\,\,\,\,\,\,\,\,\psi'_{a}} & \N}
\]

and a direct calculation shows that $\Phi$ and $\Phi'$ satisfy the
following formula: \begin{equation}
\Phi'(a)(V)=t_{F}\Phi(a)(t_{F}^{-1}Vt_{F})t_{F}^{-1}\label{phis}\end{equation}
 where $V$ is any element in $\Ga$.

Another way to write Equation \ref{phis} is $\Phi'=c_{t_{F}}\Phi c_{t_{F}}^{-1}$,
where by $c_{t_{F}}$ we mean the automorphism of $\Ga$ of conjugation
by $t_{F}$. In other words- this shows that we have some freedom
in choosing $\Phi$, and if we change $\Phi$ in the above fashion,
we will still get equivalent categories.

Consider now the case where also $\Phi=\Phi'$. This means that for
every $a\in H$ the element $t_{F}\Phi(a)(t_{F})^{-1}$ is central
in $\Ga$. A direct calculation shows that the function $r$ defined
by $r(a)=t_{F}\Phi(a)(t_{F})^{-1}$ is a one cocycle with values in
$Z(\Ga)$, and that $v/v'=r$. Notice in particular that by choosing
arbitrary $t_{F}\in Z(\Ga)$ we see that cohomologous solutions to
the second obstruction will give us equivalent categories. However,
we see that more is true, and it might happen that non cohomologous
$v$ and $v'$ will define equivalent categories.

Last, if the situation is that $t_{F}=\Phi(a)(t_{F})$ for every $a\in H$,
we will have the same $(\N,H,\Phi,v)$, but $\beta$ might be different.
We have seen that if $\beta$ and $\beta'$ are cohomologous they
will define equivalent categories, but it might happen that noncohomologous
$\beta$ and $\beta'$ will define equivalent categories as well.

\section{Functor categories \label{sec:Equivaraintization}}

In this section we are going to describe the category of functors
between module categories over an extension in terms of module categories
over the trivial component of the extension. We prove a categorical
analogue of Mackey's Theorem and we give a criterion for an extension
to be group theoretical. In addition, given that $\qC$ is a $G$-extension
of $\qD$, we describe the category $Fun_{\C}(\M_{1},\M_{2})$ of
$\qC$-module functors as an equivariantization of the category $Fun_{\D}(\M_{1},\M_{2})$
of $\qD$-module functors with respect to $G$.

\subsection{Mackey's Theorem for module categories}

Let $\qC$ be a $G$-extension of $\qD$. For any subset $S\subseteq G$
denote the subcategory $\mbox{\ensuremath{\bigoplus}}_{g\in S}\qC_{g}$
by $\qC_{S}$. If $S$ is a subgroup of $G$ then $\qC_{S}$ is a
fusion subcategory. Let $H$ and $K$ be subgroups of $G$ and let
$\nmodule$ be a $\qC_{K}$-module category. We prove now a categorical
version of Mackey's Theorem.
\begin{thm}\label{mackey}
$(\qC\boxtimes_{\qC_{K}}\nmodule)_{|\qC_{H}}\cong\bigoplus_{HgK}\qC_{H}\boxtimes_{\qC_{H^{g}}}\nmodule^{g}$,
where $H^{g}=H\cap gKg^{-1}$ and $\nmodule^{g}=(\qC_{gK}\boxtimes_{\qC_{K}}\nmodule)_{|H^{g}}$
is $\qC_{H^{g}}$-module category and the sum is over all the double
cosets.\end{thm}
\begin{proof}
First, consider the transitive $H\times K^{op}$-action on $HgK$.
The stabilizer of $g$ is $\{(gkg^{-1},k^{-1})|k\in K,gkg^{-1}\in H\}$.
Hence, $\qC_{HgK}$ is isomorphic to $\qC_{H}\boxtimes_{\qC_{H^{g}}}\qC_{gK}$
as $(\qC_{H},\qC_{K})$-bimodule category. Next, $(\qC\boxtimes_{\qC_{K}}\nmodule)_{|\qC_{H}}\cong\bigoplus_{HgK}\qC_{HgK}\boxtimes_{\qC_{K}}\nmodule$
where the sum is over all the double cosets. Finally $\qC_{HgK}\boxtimes_{\qC_{K}}\nmodule\cong(\qC_{H}\boxtimes_{\qC_{H^{g}}}\qC_{gK})\boxtimes_{\qC_{K}}\nmodule$$\cong\qC_{H}\boxtimes_{\qC_{H^{g}}}(\qC_{gK}\boxtimes_{\qC_{K}}\nmodule)$$=\qC_{H}\boxtimes_{\qC_{H^{g}}}\nmodule^{g}.$\end{proof}
\begin{rem}
The above theorem could be stated in the original Mackey's Theorem
language, namely $res_{H}^{G}ind_{K}^{G}(\nmodule)\cong\bigoplus_{HgK}ind_{H^{g}}^{H}res_{H^{g}}^{K}(\nmodule^{g})$.
One notices that the proof of the theorem uses only basic consideration
about double cosets.
\end{rem}

\subsection{Functor categories}

Assume that we have two module categories $\M_{1}=\M(\N,H,\Phi,v,\beta)$,
and $\M_{2}=\M(\N',H',\Phi',v',\beta')$. Let us denote $H'$ by $K$.
Our goal is to calculate $Fun_{\qC}(\mmodule_{1},\mmodule_{2})$ in
terms of functor categories over $\qD$. We have \[
Fun_{\qC}(\mmodule_{1},\mmodule_{2})=Fun_{\qC}(\qC\boxtimes_{\qC_{H}}\nmodule,\qC\boxtimes_{\qC_{K}}\nmodule').\]
 By Frobenious reciprocity \[
Fun_{\qC}(\qC\boxtimes_{\qC_{H}}\nmodule,\qC\boxtimes_{\qC_{K}}\nmodule')\cong Fun_{\qC_{H}}(\nmodule,(\qC\boxtimes_{\qC_{K}}\nmodule')_{|\qC_{H}}).\]
 Since a module category is, by definition, a semisimple category
every functor has both a left adjoint and a right adjoint. Taking
left adjoints (right adjoints) gives us an equivalence of the corresponding
functor categories.

Thus we obtain the following equivalence by taking left adjoints \[
Fun_{\qC_{H}}(\nmodule,(\qC\boxtimes_{\qC_{K}}\nmodule')_{|\qC_{H}})\cong Fun_{\qC_{H}}((\qC\boxtimes_{\qC_{K}}\nmodule')_{|\qC_{H}},\nmodule).\]
 By Mackey's Theorem for module categories we have \[
Fun_{\qC_{H}}((\qC\boxtimes_{\qC_{K}}\nmodule')_{|\qC_{H}},\nmodule)\cong Fun_{\qC_{H}}(\bigoplus_{HgK}\qC_{H}\boxtimes_{\qC_{H^{g}}}\nmodule'^{g},\nmodule)\]
 and \[
Fun_{\qC_{H}}(\bigoplus_{HgK}\qC_{H}\boxtimes_{\qC_{H^{g}}}\nmodule'^{g},\nmodule)\cong\bigoplus_{HgK}Fun_{\qC_{H^{g}}}(\nmodule'^{g},\nmodule_{|\qC_{H^{g}}}).\]
 Finally, by taking right adjoints, we end up with the following 

\begin{prop}
\label{pro:ext-fun-cat}In the above notations \[
Fun_{\qC}(\mmodule_{1},\mmodule_{2})\cong{\bigoplus_{HgK}Fun_{\qC_{H^{g}}}(\nmodule_{|\qC_{H^{g}}},\nmodule'^{g}).}\]

\end{prop}

\subsection{A criterion for an extension to be group theoretical}

Let $\qC$ be a fusion category. Recall that a $\qC$-module category
$\mmodule$ is called \textit{pointed} if $\cmdual$, the dual category
with respect to $\mmodule$ , is pointed. We say that $\qC$ is group
theoretical in case $\qC$ has a pointed module category. As can easily
be seen, $\qC$ is pointed if and only if it has an indecomposable
module category $\N$ such that any simple $\C$-linear functor $F:\N\rightarrow\N$
is invertible.

We now prove a criterion for an extension category to be group theoretical.

\begin{thm}
Let $\qC$ be a $G$-extension of $\qD$. $\qC$ is group theoretical
if and only if $\qD$ has a pointed module category $\nmodule$ which
is $G$-stable, namely, for every $g\in G$, $\nmodule\cong\qC_{g}\boxtimes_{\qD}\nmodule$.\end{thm}
\begin{proof}
Suppose $\nmodule$ is a pointed $G$-stable $\qD$-module category.
Consider $\M=\qC\boxtimes_{\qD}\nmodule$. By Frobenious reciprocity
we have \[
Fun_{\qC}(\mmodule,\mmodule)\cong\oplus_{g\in G}Fun_{\qD}(\nmodule,\qC_{g}\dtimes\nmodule).\]
 Since for any $g\in G$ it holds that $\C_{g}\dtimes\N\cong\N$ and
since all simple functors in $Fun_{\qD}(\N,\N)$ are invertible, we
see that the same happens in $Fun_{\qC}(\M,\M)$, that is- $\M$ is
pointed over $\C$ and $\C$ is group theoretical.

Conversely, suppose that $\qC$ is group theoretical and suppose $\mmodule$
is an indecomposable pointed $\qC$-module category. We thus know
that any simple functor $F:\M\rightarrow\M$ is invertible. We also
know that there is a subgroup $H<G$ and an indecomposable $\C_{H}$-module
category $\N$ such that $\M\cong\C\stimes_{\C_{H}}\N=\oplus_{gH\in G/H}\C_{g}\stimes_{\D}\N$
Since $\M$ is indecomposable, it is easy to see that for every $g\in G$
there is some simple $\C$-endofunctor $F:\M\rightarrow\M$ such that
$F(\N)\subseteq\C_{g}\stimes_{\D}\N$. But such a functor must be
invertible, and it follows that $F$ induces an equivalence of $\D$
module categories $\N\cong\C_{g}\stimes_{\D}\N$. Thus $\N$ is $G$-invariant.

Next, we would like to prove that $\D_{\N}^{*}$ is pointed. By Frobenius
reciprocity we have $Fun_{\C}(\M,\M)\cong\oplus_{gH\in G/H}Fun_{\C_{H}}(\N,\C_{g}\stimes\N)$
Thus the category $Fun_{\C_{H}}(\N,\N)$ is a sub-fusion category
of the pointed category $Fun_{\C}(\M,\M)$ and is therefore pointed.
We have a forgetful functor $Fun_{\C_{H}}(\N,\N)\rightarrow Fun_{\D}(\N,\N)$
which is known to be onto (see Proposition 5.3 of \cite{ENO}). This
implies that $Fun_{\D}(\N,\N)$ is pointed, as required.
\end{proof}

\begin{rem}
The above criterion is actually equivalent to the one given in Corollary 3.10 of \cite{GNN},
namely, $\qC$ is group theoretical if and only if $\mathcal{Z}(\qD)$
contains a $G$-stable Lagrangian subcategory.
In order to explain why the two conditions are equivalent,
recall first the definitions of a Lagrangian subcategory and of the action of $G$ on $\mathcal{Z}(\qD)$.
A Lagrangian subcategory of $\mathcal{Z}(\qD)$ is a subcategory $\mathcal{E}$
such that $\mathcal{E}'=\mathcal{E}$ (see Section 3.2 of \cite{DGNO} for the definition of ').
The action of $G$ on $\mathcal{Z}(\qD)$ is defined as follows: the
center $\mathcal{Z}(\qD)$ can be considered as $Fun_{\qD\boxtimes\qD^{op}}(\qD,\qD)$,
the category of $\qD$-bimodule endofunctors of $\qD$. Given an element
$g\in G$ and a $\qD$-bimodule functor $F:\qD\rightarrow\qD$, the
functor $g(F):\qD\rightarrow\qD$ is defined via \[
\xymatrix{\qD\ar[r]^{\cong\,\,\,\,\,\,\,\,\,\,\,\,\,\,\,\,\,\,\,\,\,\,\,} & \qC_{g}\dtimes\qD\dtimes\qC_{g^{-1}}\ar[rr]^{1\dtimes F\dtimes1} &  & \qC_{g}\dtimes\qD\dtimes\qC_{g^{-1}}\ar[r]^{\,\,\,\,\,\,\,\,\,\,\,\,\,\,\,\,\,\,\,\,\,\,\,\cong} & \qD}
.\]
In Theorem 4.66 of \cite{DGNO} it was proved that there is an equivalence between Lagrangian subcategory of $\mathcal{Z}(\qD)$
and pointed $\qD$ module categories. Since $G$ acts on the center of $\qD$, it also acts on the set of
Lagrangian subcategories of $\mathcal{Z}(\qD)$.
Let $\nmodule$ be a $\qD$-module category and let $\mathcal{L}$
be the corresponding Lagrangian subcategory of $\mathcal{Z}(\qD)$.
From the above definition of the $G$ action on $\mathcal{Z}(\qD)$
it is possible to see that for any $g\in G$, the lagrangian subcategory which corresponds to $g\cdot\nmodule$
is $g\cdot\mathcal{L}$.
Therefore, $\mathcal{Z}(\qD)$ admit a $G$-stable Lagrangian subcategory
if and only if $\qD$ has a pointed $G$-stable module category.
\end{rem}

\subsection{Functor categories as equivariantizations}

In this subsection we shall describe the category of $\C$-module
functors between the module categories $\M(\N,H,\Phi,v,\beta)$ and
$\M(\N',H',\Phi',v',\beta')$. For simplicity we shall denote these
categories as $\M_{1}$ and $\M_{2}$ respectively.

Consider the category $Fun_{\D}(\M_{1},\M_{2})$. This is a $k$-linear
category which by Theorem 2.16 of \cite{ENO} is semisimple.
\begin{lem}
There is a natural $G$-action on $Fun_{\qD}(\mmodule_{1},\mmodule_{2})$
induced by the structure of $\qC$-module categories on $\mmodule_{1}$
and $\mmodule_{2}$.\end{lem}
\begin{proof}
There are $\qD$-module equivalences $\psi_{g}:\qC_{g}\dtimes\mmodule_{1}\cong\mmodule_{1}$
and $\phi_{g}:\qC_{g}\dtimes\mmodule_{2}\cong\mmodule_{2}$, for every
$g\in G$, defined by the $\qC$-module structure on $\mmodule_{1}$
and $\mmodule_{2}$. Let $F:\mmodule_{1}\rightarrow\mmodule_{2}$
be a $\qD$-module morphism, we define $g\cdot F$ to be the following
functor \[
\xymatrix{\mmodule_{1}\ar[r]^{\psi_{g}^{-1}\,\,\,\,\,\,\,\,} & \qC_{g}\dtimes\mmodule_{1}\ar[r]^{Id_{\qC_{g}}\dtimes F} & \qC_{g}\dtimes\mmodule_{2}\ar[r]^{\,\,\,\,\,\,\,\,\phi_{g}} & \mmodule_{2}}
.\]
 One can easily check that this defines an action of the group $G$
on the category $Fun_{\D}(\mmodule_{1},\mmodule2)$ in the sense of
\cite{EG} 

\end{proof}
Since we have a $G$-action on $Fun_{\qD}(\mmodule_{1},\mmodule_{2})$,
we can talk about the equivariantization $Fun_{\qD}(\mmodule_{1},\mmodule_{2})^{G}$.
By definition, an object in $Fun_{\qD}(\mmodule_{1},\mmodule_{2})^{G}$
is a pair $(F,\{T_{g}\}_{g\in G})$, where $T_{g}:g\cdot F\rightarrow F$
are natural equivalences which satisfy a certain coherence condition
(for the exact definition, see \cite{EG}). Let $F:\mmodule_{1}\rightarrow\mmodule_{2}$
be a $\qD$-module functor.

To give $F:\M_{1}\rightarrow\M_{2}$ a structure of a $\C$-module
functor is the same thing as to give, for every $g\in G$, a natural
isomorphism between the functors $\C_{g}\stimes_{\D}\M_{1}\rightarrow\M_{1}\stackrel{F}{\rightarrow}\M_{2}$
and $\C_{g}\stimes_{\D}\M_{1}\stackrel{1\stimes F}{\rightarrow}\C_{g}\stimes_{\D}\M_{2}\rightarrow\M_{2}$.
It can easily be seen that this is equivalent to give $F$ a structure
of an object in the equivariantization category.

Let us conclude this discussion by the following
\begin{prop}
The category $Fun_{\C}(\M_{1},\M_{2})$ is equivalent to the equivariantization
$Fun_{\D}(\M_{1},\M_{2})^{G}$ of the category $Fun_{\D}(\M_{1},\M_{2})$
with respect to the aforementioned $G$-action.\end{prop}
\begin{rem}
Let $\mmodule$ be an indecomposable $\qC$-module category. Although
$\qC_{\mmodule}^{*}\triangleq Fun_{\qC}(\mmodule,\mmodule)$ is a
fusion category, $Fun_{\qD}(\mmodule,\mmodule)$ is, in general, only
a multifusion category because $\mmodule$ might be decomposable
as $\qD$-module category. Equivariantization has only been defined
in the context of fusion categories. However, the definition in context
of multifusion categories is mutatis mutandis. Notice that is case
of multifusion equivariantization we don't always have the $Rep(G)$
subcategory supported on the trivial object.
\end{rem}
In the next section we will give an intrinsic description of the functor
categories, as categories of bimodules.

\section{An intrinsic description by algebras and modules\label{sec:An-intrinsic-description}}

The goal of this section is to explain more concretely the action
of the grading group on indecomposable module categories, the action
of the grading group on $Aut_{\D}(\N)$, the obstructions and their
solutions.

In \cite{O1} Ostrik showed that any indecomposable module category
over a fusion category $\C$ is equivalent as a module category to
the category $Mod_{\C}-A$ for some semisimple indecomposable algebra
$A$ in $\C$. In this section we will realize all the objects described
in the previous sections by using algebras and modules inside $\C$.
As before, we assume that $\C=\bigoplus_{g\in G}\C_{g}$, we denote
$\C_{1}$ by $\D$ and $Aut_{\D}(\N)$ by $\Gamma$.

\subsection{The action of $G$ on indecomposable module categories}

Assume that $A$ is a semisimple indecomposable algebra inside $\D$.
Let $\N=Mod_{\D}-A$ be the category of right $A$-modules inside
$\D$. We denote by $Mod_{\C_{g}}-A$ the category of $A$-modules
with support in $\C_{g}$. We claim the following:
\begin{lem}
We have an equivalence of $\D$-module categories $\C_{g}\stimes_{\D}\N\equiv Mod_{\C_{g}}-A$.\end{lem}
\begin{proof}
We have already seen in Section \ref{sec:Preliminaries} that we have
an equivalence of $\C$-module categories \[
\C\stimes_{\D}\N\equiv Mod_{\C}(A)\]
 which is given by $X\stimes M\mapsto X\otimes M$. As a $\D$-modules
category, the left hand category decomposes as $\bigoplus_{g\in G}\C_{g}\stimes_{\D}\N$
and the right hand category decomposes as $\bigoplus_{g\in G}Mod_{\C_{g}}-A$.
It is easy to see that the above equivalence translates one decomposition
into the other, and therefore the functor $\C_{g}\stimes_{\D}\N\rightarrow Mod_{\C_{g}}-A$
given by $X\stimes M\mapsto X\otimes M$ is an equivalence of $\D$-module
categories.
\end{proof}
Next, we understand how we can describe functors by using bimodules.
\begin{lem}
\label{lem:Functors}Let $\N=Mod_{\D}-A$ and $\N'=Mod_{\D}-A'$,
and let $g\in G$. Then every functor $F:\N\rightarrow\C_{g}\stimes_{\D}\N'$
is of the form $F(T)=T\otimes_{A}Y$ for some $A-A'$ bimodule $Y$
with support in $\C_{g}$, here we identify $\C_{g}\stimes_{\D}\N'$
with $Mod_{\C_{g}}-A'$ as above.\end{lem}
\begin{proof}
The proof follows the lines of the remark after Proposition 2.1 of
\cite{O2}. We simply consider $F(A)$. The multiplication map $A\otimes A\rightarrow A$
gives us a map $A\otimes F(A)\rightarrow F(A)$, thus equipping $F(A)$
with a structure of a left $A$-module. We now see that $F(A)$ is
indeed an $A-A'$ bimodule. Since the category $\N$ is semisimple
the functor $F$ is exact. Since every object in $\N$ s a quotient
of an object of the form $X\otimes A$ for some $X\in\C$, we see
that $F$ is given by $F(T)=T\otimes_{A}F(A)$.\end{proof}
\begin{rem*}
Notice that by applying the (2-)functor $\C_{g^{-1}}\stimes_{\D}-$
we see that every functor $\C_{g}\stimes_{\D}\N'\rightarrow\N$ is
given by tensoring with some $A'-A$ bimodule with support in $C_{g^{-1}}$.
\end{rem*}

\subsection{The outer action of $H$ on the group $Aut_{\D}(\N)$. The first
two obstructions}

Assume, as in the rest of the paper, that we have a subgroup $H<G$
and a module category $\N=Mod_{\D}-A$, and assume that $F_{h}:\nmodule\cong\qC_{h}\dtimes\nmodule$
for every $h\in H$. It follows from Lemma \ref{lem:Functors} that
this equivalence is of the form $F_{h}(M)=M\otimes_{A}A_{h}$ for
some $A-A$ bimodule $A_{h}$ with support in $\C_{h}$. The fact
that this functor is an equivalence simply means that the bimodule
$A_{h}$ is an invertible $A-A$ bimodule. In other words- there is
another $A-A$ bimodule $B_{h}$ (whose support will necessary be
in $\C_{h^{-1}}$) such that $A_{h}\otimes_{A}B_{h}\cong B_{h}\otimes_{A}A_{h}\cong A$.
By Lemma \ref{lem:Functors} we can identify the group $\Gamma=Aut_{\D}(\N)$
with the group of isomorphisms classes of invertible $A-A$ bimodules
with support in $\D$.

Denote by $\Lambda$ the group of isomorphisms classes all invertible
$A-A$ bimodules with support in $\qC_{H}$. Since every invertible
$A-A$ bimodule is supported on a single grading component, we have
a map $p:\Lambda\rightarrow H$ which assigns to an invertible $A-A$
bimodule the graded component it is supported on. We thus have a short
exact sequence \begin{equation}
1\rightarrow\Gamma\rightarrow\Lambda\rightarrow H\rightarrow1.\label{sequence}\end{equation}

Using this sequence, we can understand the outer action of $H$ on
$Aut_{\D}(\N)$, and the first and the second obstruction. The outer
action is given in the following way: for $h\in H$, choose an invertible
$A-A$ bimodule $A_{h}$ with support in $\C_{h}$. Choose an inverse
to $A_{h}$ and denote it by $A_{h}^{-1}$. Then the action of $h\in H$
on some invertible bimodule $M$ with support in $\D$ is the following
conjugation: \[
h\cdot M=A_{h}\otimes_{A}M\otimes_{A}A_{h}^{-1}.\]
 This action depends on the choice we made of the invertible bimodule
$A_{h}$.

The first obstruction is the possibility to lift this outer action
to a proper action. In other words, it says that we can choose the
$A_{h}$'s in such a way that conjugation by $A_{h}\otimes A_{h'}$
is the same as conjugation by $A_{hh'}$, or in other words, in such
a way that for every $h,h'\in H$, the invertible bimodule \[
B_{h,h'}=A_{h}\otimes_{A}A_{h'}\otimes_{A}A_{hh'}^{-1}\]
 will be in the center of $\Gamma$ (again- we identify $\Gamma$
with the group of invertible bimodules with support in $\D$). A solution
for the first obstruction will be a choice of a set of such bimodules
$A_{h}$.

The second obstruction says that the cocycle $(h,h')\mapsto B_{h,h'}$
is trivial in $H^{2}(H,Z(Aut_{\D}(\N))$. This simply says that we
can change $A_{h}$ to be $A_{h}\otimes_{A}D_{h}$ for some $D_{h}\in Z(Aut_{\D}(\N))$,
in such a way that \[
(A_{h}\otimes_{A}D_{h})\otimes_{A}(A_{h'}\otimes_{A}D_{h'})\otimes_{A}(A_{hh'}\otimes_{A}D_{hh'})^{-1}\cong A\]
 as $A$-bimodules. A solution for the second obstruction will be
a choice of such a set $D_{h}$ of bimodules.

It is easier to understand the first and the second obstruction together:
we have one big obstruction- the sequence \ref{sequence} should split,
and we need to choose a splitting. First, if the sequence splits,
then we can lift the outer action into a proper action, and we need
to choose such a lifting. Then, the obstruction to the splitting \textit{with
the chosen action} is given by a two cocycle with values in the center
of $\Gamma$. Thus, a solution for both the first and the second obstruction
will be a choice of bimodules $A_{h}$ for every $h\in H$ such that
the support of $A_{h}$ is in $\C_{h}$ and such that $A_{h}\otimes_{A}A_{h'}\cong A_{hh'}$
for every $h,h'\in H$. Following the line of Section \ref{sec:The-isomorphism-condition},
we see that we are interested in splittings only up to conjugation
by an element of $\Gamma$.

\subsection{The third obstruction}

Assume then that we have a set of bimodules $A_{h}$ as in the end
of the previous subsections. We would like to understand now the third
obstruction.

Recall that we are trying to equip $\N$ with a structure of a $\C_{H}$-module
category. By Ostrik's Theorem (see \cite{O1}), there is an object
$\N\in\N$ such that $A\cong\underline{Hom}_{\D}(N,N)$ where by $\underline{Hom}_{\D}$
we mean the internal Hom of $\N$, where we consider $\N$ as a $\D$-module
category. So far we gave equivalences $F_{h}:\N\rightarrow\C_{h}\stimes_{\D}\N$.
If $\N$ were a $\C_{H}$-module category via the choices of these
equivalences, then the internal $\C_{H}$-Hom, $\tilde{A}=\underline{Hom}_{\C_{H}}(N,N)$
would be \[
\tilde{A}=\bigoplus_{h\in H}A_{h}.\]

We thus see that to give on $\N$ a structure of a $\C_{H}$-module
category is the same as to give on $\tilde{A}$ a structure of an
associative algebra. For every $h,h'\in H$, choose an isomorphism
of $A-A$ bimodules $A_{h}\otimes_{A}A_{h'}\rightarrow A_{hh'}$.
Notice that since these are invertible $A-A$ bimodules, there is
only one such isomorphism up to a scalar.

Now for every $h,h',h''\in H$, we have two isomorphisms $(A_{h}\otimes_{A}A_{h'})\otimes_{A}A_{h''}\rightarrow A_{hh'h''}$,
namely \[
(A_{h}\otimes_{A}A_{h'})\otimes_{A}A_{h''}\rightarrow A_{hh'}\otimes_{A}A_{h''}\rightarrow A_{hh'h''}\]
 and \[
(A_{h}\otimes_{A}A_{h'})\otimes_{A}A_{h''}\rightarrow A_{h}\otimes_{A}(A_{h'}\otimes_{A}A_{h''})\rightarrow A_{h}\otimes_{A}A_{h'h''}\rightarrow A_{hh'h''}.\]
 This two isomorphisms differ by a scalar $b(h,h',h'')$. The function
$(h,h',h'')\mapsto b(h,h',h'')$ is a three cocycle which is the third
obstruction. A solution to the third obstruction will thus be a choice
of isomorphisms $A_{h}\otimes_{A}A_{h'}\rightarrow A_{hh'}$ which
will make $\tilde{A}$ an associative algebra. Once we have such a
choice, we can change it by some two cocycle to get another solution.

\subsection{Functor categories}

We end this section by giving an intrinsic description of functor
categories. Assume that we have two module categories $\M_{1}=\M(\N,H,\Phi,v,\beta)$,
and $\M_{2}=\M(\N',H',\Phi',v',\beta')$. Let us denote $H'$ by $K$.
As we have seen in the previous subsections, if $\N\cong Mod_{\D}-A_{1}$
and $\N'\cong Mod_{\D}-B_{1}$, then $\M_{1}\cong Mod_{\C}-A$ and
$\M_{2}\cong Mod_{\C}-B$, where $A$ is an algebra of the form $\oplus_{h\in H}A_{h}$,
and a similar description holds for $B$.

The functor category $Fun_{\C}(\M_{1},\M_{2})$ is equivalent to the
category of $A-B$-bimodules in $\C$. Since $A$ and $B$ have a
graded structure, we will be able to say something more concrete on
this category.

Let $X$ be an indecomposable $A-B$-bimodule in $\C$. It is easy
to see that the support of $X$ will be contained inside a double
coset of the form $HgK$ for some $g\in G$. Since the bimodules $A_{h}$
and $B_{k}$ are invertible, it is easy to see that the support will
be exactly this double coset.

Consider now the $g$-component $X_{g}$ of $X$. As can easily be
seen, this is an $A_{1}-B_{1}$-bimodule. Actually, more is true.
Consider the category $\C\stimes\C^{op}$. Inside this category we
have the algebra \[
(AB)_{g}=\oplus_{x\in H\cap gKg^{-1}}A_{x}\stimes B_{g^{-1}x^{-1}g}\]
 with the multiplication defined by the restricting the multiplication
from $A\boxtimes B\in\qC\boxtimes\qC^{op}$ . The category $\C$ is
a $\C\stimes\C^{op}$-module category in the obvious way, and we have
a notion of an $(AB)_{g}$-module inside $\C$.
\begin{lem}
The category of $(AB)_{g}$-modules inside $\C$ is equivalent to
the category of $A-B$-bimodules with support in the double coset
$HgK$. \end{lem}
\begin{proof}
If $X$ is an $A-B$-bimodule with support in $HgK$, then $X_{g}$
is an $(AB)_{g}$-module via restriction of the left $A$-action and
the right $B$-action. Conversely, if $V$ is an $(AB)_{g}$-module
inside $\C$, we can consider the induced module \[
(A\stimes B)\otimes_{(AB)_{g}}V.\]
 This is an $A-B$-bimodule, and one can see that the two constructions
gives equivalences in both directions. \end{proof}
\begin{rem*}
This is a generalization of Proposition 3.1 of \cite{O2}, where the
same situation is considered for the special case that $\C=Vec_{G}^{\omega}$
and $\D=1$. Also, notice that the decomposition to double cosets is the one which appears in Theorem \cite{mackey}
\end{rem*}
In conclusion, we have the following
\begin{prop}
The functor category $Fun_{\C}(\M_{1},\M_{2})$ is equivalent to the
category of $A-B$-bimodules. Each such simple bimodule is supported
on a double coset of the form $HgK$, and the subcategory of bimodules
with support in $HgK$ is equivalent to the category of $(AB)_{g}$-modules
inside $\C$.
\end{prop}

\section{A detailed example: classification of modules categories over the
Tambara Yamagami fusion categories and their dual categories\label{sec:An-example:-classification}}

As an example of our results, we shall now describe the module categories
over the Tambara Yamagami fusion categories $\C=\mathcal{TY}(A,\chi,\tau)$
and the corresponding dual categories. Let $A$ be a finite group.
Let $R_{A}$ be the fusion ring with basis $A\cup\{m\}$ whose multiplication
is given by the following formulas: \\
 \[
g\cdot h=gh,\forall g,h\in A\]
 \[
g\cdot m=m\cdot g=m\]
 \[
m\cdot m=\sum_{g\in A}g\]
 In \cite{TY} Tambara and Yamagami classified all fusion categories
with the above fusion ring. They showed that if there is a fusion
category $\mathcal{C}$ whose fusion ring is $R_{A}$ then $A$ must
be abelian. They also showed that for a given $A$ such fusion categories
can be parameterized (up to equivalence) by pairs $(\chi,\tau)$ where
$\chi:A\times A\rightarrow k^{*}$ is a nondegenerate symmetric bicharacter,
and $\tau$ is a square root (either positive or negative) of $\frac{1}{|A|}$.
We denote the corresponding fusion category by $\C:=\mathcal{TY}(A,\chi,\tau)$.

The category $\C$ is naturally graded by $\Z_{2}=\langle\sigma\rangle$.
The trivial component is $Vec_{A}$ (with trivial associativity constraints)
and the nontrivial component, which we shall denote by $\M$, has
one simple object $m$. In \cite{ENO2} the authors described how
the Tambara Yamagami fusion categories corresponds to an extension
data of $Vec_{A}$ by the group $\Z_{2}$. We shall explain now the
classification of module categories over $\ty$ given by our parameterization.

Since $A$ is an abelian group and the associativity constraints in
$Vec_{A}$ are trivial, module categories over $Vec_{A}$ are parameterized
by pairs $(H,\psi)$ where $H<A$ is a subgroup and $\psi\in H^{2}(H,k^{*})$.
We shall denote the corresponding module category by $\mhpsi$. As
explained in Section \ref{sec:zero}, we have a natural action of
$\Z_{2}=\langle\sigma\rangle$ on the set of equivalence classes of
module categories over $Vec_{A}$. We shall describe this action in
Subsection \ref{sub:sigmaaction}.

Recall that the second component in the parameterization of a module
category is a subgroup of the grading group. If this subgroup is the
trivial subgroup, then we will just have a category which is induced
from $Vec_{A}$. It is easy to see that such categories decompose
over $Vec_{A}$ to the direct sum of two indecomposable module categories.
In that case, all the obstructions and solutions will be trivial.
If this subgroup is $\Z_{2}$ itself, we will have a $\C$-module
category structure on $\mhpsi$ for some $H$ and some $\psi$. In
that case, it must hold that $\sigma(H,\psi)=(H,\psi)$, and we may
have some nontrivial obstructions and solutions.

The rest of this section will be devoted to analyze the action of
$\sigma$ and the obstructions and their solutions (for the case in
which we have obstructions). We will also describe the relations of
our result with the result of Tamabra on fiber functors on Tamabara
Yamagami categories, and also describe the dual categories.

We would like now to describe the main result of this section. We
will split our main classification result into two proposition, according
to the subgroup of $\Z_{2}$ which appears in the parameterization.
Our first proposition follows in a straight forward way from the discussion
in the previous sections
\begin{prop}
Module categories over $\C$ whose parameterization begins with $(\mhpsi,1,\ldots)$
are the induced categories $Ind_{Vec_{A}}^{\C}(\mhpsi)$. We will
have an equivalence of $\C$ module categories $Ind_{Vec_{A}}^{\C}(\mhpsi)\cong Ind_{Vec_{A}}^{\C}(\M(H',\psi'))$
if and only if $(H,\psi)=(H',\psi')$ or if $(H,\psi)=\sigma(H'\psi')$.
\end{prop}
In order to describe the other case, we need some notations. Suppose
that $H<A$ is a subgroup which contains $H^{\nizav}$ (the subgroup perpendicular to $H$ with respect to $\chi$).
If we denote by $\bar{H}:=H/H^{\nizav}$, then $\chi$ induces a non-degenerate
symmetric bicharacter $\bar{\chi}:\bar{H}\times\bar{H}\rightarrow k^{*}$.
If $\psi\in H^{2}(H,k^{*})$ satisfies $Rad(\psi)=H^{\nizav}$ (the
definition of $rad(\psi)$ is given in Subsection \ref{sub:sigmaaction}),
then $\psi$ is the inflation of a nondegenerate two cocycle $\bar{\psi}$
on $\bar{H}$. We will usually not distinguish between $\psi$ and
$\bar{\psi}$.
\begin{prop}
\label{maintambara} For $\mhpsi$ to have a structure of a $\C$
module category, it is necessary that $\sigma(H,\psi)=(H,\psi)$.
This implies that $Rad(\psi)=H^{\nizav}<H$. If this holds, then $\C$-module
categories structures on $\mhpsi$ are parameterized by pairs $(s,\nu)$
where $s:H/H^{\nizav}\rightarrow H/H^{\nizav}$ is an involutive automorphism,
and $\nu:H/H^{\nizav}\rightarrow k^{*}$ is a function which satisfy
for every $a,b\in H/H^{\nizav}$ \[
\bar{\chi}(a,b)=\psi(s(a),b)/\psi(b,s(a))\]
 \[
\del\nu(a,b)=\psi(a,b)/\psi(s(b),s(a))\]
 \[
\nu(a)\nu(s(a))=1\]
 \[
sign(\sum_{s(a)=a}\nu(a))=sign(\tau)\]
 Two such pairs $(s,\nu)$ and $(s',\nu')$ will give equivalent module
category structures on $\mhpsi$ if and only if $s=s'$ and there
exist a character $\phi:H/H^{\nizav}\rightarrow k^{*}$ such that
$\nu(h)/\nu'(h)=\eta(h)/\eta(s(h))$.
\end{prop}

\subsection{The action of $\sigma$ on indecomposable module categories and representations
of twisted abelian group algebras\label{sub:sigmaaction}}

Recall that the $Vec_{A}$ module category $\N=\mhpsi$ is the category
of right modules over the algebra $k^{P\psi}H$ inside $Vec_{A}$.
We would like to understand the $Vec_{A}$ module category $\M\stimes_{Vec_{A}}\N$.

As explained in Section \ref{sec:An-intrinsic-description}, this
module category can be described as the category of right $k^{\psi}H$-modules
with support in the category $\M$, the nontrivial grading component
of $\C$. A $k^{\psi}H$-module with support in $\M$ is of the form
$m\otimes V$ where $V$ is a vector space which is a $k^{\psi}H$-module
in the usual sense. So the category $\M\stimes_{Vec_{A}}\N$ is equivalent,
at least as an abelian category, to the category of $k^{\psi}H$-modules
in $Vec$.

We would like to describe $\M\stimes_{Vec_{A}}\N$ as a module category
of the form $\M(H',\psi')$ for some $H'<A$ and some two cocycle
$\psi'\in H^{2}(H',k^{*})$. In order to do so, we begin by describing
the simple $k^{\psi}H$ modules in $Vec$ (they will correspond to
the simple objects in $M\stimes_{Vec_{A}}\N$).

Let $k^{\psi}H=\oplus_{h\in H}U_{h}$. The multiplication in $k^{\psi}H$
is given by the rule $U_{h}U_{k}=\psi(h,k)U_{hk}$. Denote by $R=Rad(\psi)$
the subgroup of all $h\in H$ such that $U_{h}$ is central in $k^{\psi}H$.

As the field $k$ is algebraically closed of characteristic zero and
$H$ is abelian, the data that stored in the cocycle $\psi$ is simply
the way in which the $U_{h}$'s commute. More precisely- let us define
the following alternating form on $H$: \[
\xi_{\psi}(a,b)=\psi(a,b)/\psi(b,a).\]

It turns out (see \cite{T}) that the assignment $\psi\mapsto\xi_{\psi}$
depends only on the cohomology class of $\psi$, and that it gives
a bijection between $H^{2}(H,k^{*})$ and the set of all alternating
forms on $H$. The elements of $R$ can be described as those $h\in H$
such that $\xi_{\psi}(h,-)=1$. As can easily be seen, $\xi_{\psi}$
is the inflation of an alternating form on $H/R$. It follows easily
that $\psi$ is the inflation of a two cocycle $\bar{\psi}$ on $H/R$.

It can also be seen that $\xi_{\bar{\psi}}$ is nondegenerate on $H/R$
and that $k^{\bar{\psi}}H/R\cong M_{n}(k)$ where $n=\sqrt{|H/R|}$.
It follows that $k^{\bar{\psi}}H/R$ has only one simple module (up
to isomorphism) which we shall denote by $V_{1}$ (i.e, $\bar{\psi}$
is non degenerate on $H/R$). By inflation, $V_{1}$ is also a $k^{\psi}H$-module.
Let $\zeta$ be a character of $H$, and let $k^{\zeta}$ be the corresponding
one dimensional representation of $H$. Then $k^{\zeta}\otimes V_{1}$
is also a simple module of $k^{\psi}H$, where $H$ acts diagonally.
It turns out that these are all the simple modules of $k^{\psi}H$,
and that $V_{\zeta_{1}}\cong V_{\zeta_{2}}$ if and only if the restrictions
of $\zeta_{1}$ and $\zeta_{2}$ to $R$ coincide.

The simple modules of $k^{\psi}H$ are thus parameterized by the characters
of $R$ (we use here the fact that the restriction from the character
group of $H$ to that of $R$ is onto). For every character $\zeta$
of $R$, we denote by $V_{\zeta}$ the unique simple module of $k^{\psi}H$
upon which $R$ acts via the character $\zeta$. So the simple $k^{\psi}H$-modules
with support in $\M$ are of the form $m\otimes V_{\zeta}$.

In order to understand the structure of $\M\stimes_{Vec_{A}}\N$ as
a $Vec_{A}$ module category, let us describe $V_{a}\otimes(m\otimes V_{\zeta})$
for $a\in A$. It can easily be seen that this is also a simple module,
so we just need to understand via which character $R$ acts on it.
Using the associativity constraints in $\ty$, we see that for $v\in V_{\zeta}$
and $r\in R$ we have \[
(V_{a}\otimes m\otimes v)\cdot U_{r}=\chi(a,r)V_{a}\otimes(m\otimes v\cdot U_{r})=\chi(a,r)\zeta(r)V_{a}\otimes m\otimes v.\]
 This means that $V_{a}\otimes(m\otimes V_{\zeta})=m\otimes V_{\zeta\chi(a,-)}$.
So the stabilizer of $V_{\zeta}$ is the subgroup of all $a\in A$
such that $\chi(a,r)=1$ for all $r\in R$, i.e., it is $R^{\nizav}$.
It follows
that $\M\stimes_{Vec_{A}}\N$ is equivalent to a category of the form
$\M(R^{\nizav},\tilde{\psi})$. Where $\tilde{\psi}$ is some two
cocycle.

Let us figure out what is $\tilde{\psi}$. If $a\in R^{\nizav}$,
then the restriction of $\chi(a,-)$ to $H$ is a character which
vanishes on $R$. Therefore, there is a unique (up to multiplication
by an element of $R$) element $t_{a}\in H$ such that $\xi_{\psi}(t_{a},-)=\chi(a,-)$.
It follows that there is an isomorphism $r_{a}:V_{a}\otimes(m\otimes V_{1})\rightarrow m\otimes V_{1}$
which is given by the formula $V_{a}\otimes(m\otimes v)\mapsto m\otimes(v\cdot U_{t_{a}})$.
Now for every $a,b\in R^{\nizav}$, $\tilde{\psi}(a,b)$ should make
the following diagram commute:

\[
\xymatrix{(V_{a}\otimes V_{b})\otimes(m\otimes V_{1})\ar[d]^{r_{ab}}\ar[r] & V_{a}\otimes(V_{b}\otimes(m\otimes V_{1}))\ar[r]^{\,\,\,\,\,\,\,\,\,\,\, r_{b}} & V_{a}\otimes(m\otimes V_{1})\ar[d]^{r_{a}}\\
m\otimes V_{1}\ar[rr]^{\tilde{\psi}(a,b)} &  & m\otimes V_{1}}
\]


An easy calculation shows that this means that $\tilde{\psi}(a,b)=\psi(t_{b},t_{a})$.
We thus have the following result:
\begin{lem}
We have $\sigma\cdot\mhpsi\equiv\M(R^{\nizav},\tilde{\psi})$ where
$R$ is the radical of $\psi$ and $\tilde{\psi}$ is described above.
\end{lem}
Suppose now that $\sigma\cdot\mhpsi\equiv\mhpsi$. This implies that
$Rad(\psi)=H^{\nizav}$. The bicharacter $\chi$ defines by restriction
a pairing on $H\times H$, and by dividing out by $H^{\nizav}$, we
get a nondegenerate symmetric bicharacter $\bar{\chi}:H/H^{\nizav}\times H/H^{\nizav}\rightarrow k^{*}$.
It is easy to see that the assignment $h\mapsto t_{h}$ that was described
above induces an automorphism $s$ of $H/H^{\nizav}$ which satisfies
\begin{equation}
\bar{\chi}(a,b)=\xi_{\bar{\psi}}(s(a),b)\label{sandpsi}.\end{equation}
The fact that $\tilde{\psi}=\psi$ means that $\xi_{\bar{\psi}}(s(b),s(a))=\xi_{\bar{\psi}}(a,b)$.
Equivalently, this means that $\bar{\chi}(a,b)=\xi_{\bar{\psi}}(s(a),b)=\xi_{\bar{\psi}}(s(b),s^{2}(a))=\bar{\chi}(b,s^{2}(a))$
and since $\bar{\chi}$ is nondegenerate, this is equivalent to the
fact that $s^{2}=Id$.

In summary:
\begin{lem}
\label{lem:invariant-ty-modules}We have $\sigma\cdot\mhpsi\equiv\mhpsi$
if and only if the following two conditions hold:

1.$H^{\nizav}<H$.

2.There is an automorphism $s$ of order 2 of $H/H^{\nizav}$ such
that $(a,b)\mapsto\bar{\chi}(s(a),b)$ is an alternating form, and
the inflation of this alternating form to $H$ is $\xi_{\psi}$.
\end{lem}

\subsection{The vanishing of the first obstruction and invertible bimodules with
support in $Vec_{A}$\label{sub:abimods}}

Assume now that we have a module category $\mhpsi$ such that $\sigma\cdot\mhpsi\equiv\mhpsi$.
We would like to describe all module categories whose classification
data begins with $(\mhpsi,\Z_{2},\ldots)$. In other words- we would
like to describe all possible ways (if any) to furnish a structure
of a $\C$ module category on $\mhpsi$.

So let $s$ be an automorphism as in Lemma \ref{lem:invariant-ty-modules}.
In order to explain the first obstruction for furnishing a $\ty$-module
category structure on $\mhpsi$, we need to consider the group of
invertible $k^{\psi}H$-bimodules in $\ty$. As we have seen in Section
\ref{sec:An-intrinsic-description}, such an invertible bimodule with
support in $Vec_{A}$ ($\M$) corresponds to a functor equivalence
$F:\N\rightarrow\N$ ($F:\M\stimes_{Vec_{A}}\N\rightarrow\N$). The
functor is given by tensoring with the invertible bimodule.

Let us first classify invertible $k^{\psi}H$-bimodules with support
in $Vec_{A}$. Their description was given in Ostrik's paper \cite{O2}.
We recall it briefly.

If $a\in A$ and $\lambda$ is a character on $H$, we define the
bimodule $M_{a,\lambda}$ to be \[
\oplus_{h\in H}V_{ah},\]
 where the action of $k^{\psi}H$ is given by \[
U_{h}\cdot V_{ah'}\cdot U_{h''}=\psi(h,h')\lambda(h)\psi(hh',h'')V_{ahh'h''}.\]

Choose now coset representatives $a_{1},\ldots,a_{r}$ of $H$ in
$A$. Proposition 3.1 of \cite{O2} tells us that the modules $M_{a_{i},\lambda}$
where $i=1,\ldots r$ and $\lambda\in\hat{H}$ are all the invertible
$k^{\psi}H$ bimodules, and each invertible bimodule with support
in $Vec_{A}$ appears in this list exactly once.

By a more careful analysis we can get to the following description
of the group of invertible bimodules: we have a homomorphism $\xi:H\rightarrow\hat{H}$
given by $h\mapsto\xi_{\psi}(h,-)$. Then the group $E$ of all invertible
bimodules with support in $Vec_{A}$ can be described as the pushout
which appears in the following diagram: (see Theorem 5.2 of \cite{GN} for a more general result)

\[
\xymatrix{H\ar[d]^{\xi}\ar[r] & A\ar[d]\\
\hat{H}\ar[r] & E}
\]

The group $E$ is thus also isomorphic to the group $Aut_{Vec_{A}}(\mhpsi)$.
Notice that the group $E$ is \textit{abelian}. A solution to the
first obstruction is a lifting of the natural map (see Section \ref{sec:The-first-two})
$\Z_{2}\rightarrow Out(E)$ to a map $\Z_{2}\rightarrow Aut(E)$ But
since $E$ is abelian, $Out(E)=Aut(E)$, so this problem is trivial,
and it has only one solution. So we have a proper (and not just outer)
action of $\Z_{2}$ on $E$.

\subsection{The group of all invertible bimodules and the second obstruction}

Since $\mhpsi$ is $\sigma$-invariant, we see by Section \ref{sec:An-intrinsic-description}
that the group $\tilde{E}$ of (isomorphism classes of) invertible
$k^{\psi}H$ bimodules in $\C$ is given as an extension \[
\Sigma:1\rightarrow E\rightarrow\tilde{E}\rightarrow\Z_{2}\rightarrow1.\]
 Moreover, we have seen that the second obstruction is the cohomology
class of this extension in $H^{2}(\Z_{2},E)$, and that a solution
to the second obstruction is a splitting of this sequence, up to conjugation
by an element of $E$.

So our next goal is to understand if the sequence $\Sigma$ splits.
For this, we would like to understand the structure of the group $\tilde{E}$
better, and for this reason, we will describe now the invertible $k^{\psi}H$
bimodules with support in $\M$ (these are the elements of $\tilde{E}$
which goes to the nontrivial element in $\Z_{2}$). 
We begin by choosing such an invertible bimodule $X$ explicitly.
It should be of the form $X=m\otimes V$, where $V$ is both a left
and a right $k^{\psi}H$ module. The interaction between the left
structure and the right structure follows from the associativity constraints
and is given by the formula \begin{equation}
(U_{h}\cdot v)\cdot U_{h'}=\chi(h,h')U_{h}\cdot(v\cdot U_{h'}).\label{equation1}\end{equation}
 The fact that $X$ is invertible implies that $V$ has to be simple
as a left and as a right $\kpsih$-module. Assume that $V$ is $V_{\phi}$
from Subsection \ref{sub:sigmaaction} as a right $k^{\psi}H$ module,
where $\phi$ is some character of $H^{\nizav}$. We need to define
on $V$ a structure of a left $\kpsih$-module. By Equation \ref{sandpsi}
we know that \[
(v\cdot U_{t_{h}})\cdot U_{h'}=\chi(h,h')(v\cdot U_{h'})U_{t_{h}}.\]

By Equation \ref{equation1} and by the simplicity of $V$, we see
that this means that we must have \begin{equation}
U_{h}\cdot v=\nu(h)v\cdot U_{t_{h}}\label{equation-emo}\end{equation}
 for some set of scalars $\{\nu(h)\}_{h\in H}$. An easy calculation
shows that these scalars should satisfy the equation \[
\nu(ab)\psi(a,b)=\nu(a)\nu(b)\psi(t_{b},t_{a})\phi(t_{a}t_{b}t_{ab}^{-1})\]
 for every $a,b\in H$. In other words- \begin{equation}
\del(\nu\phi(t_{-}))=\psi(a,b)/\psi(t_{b},t_{a}).\label{equation2}\end{equation}
 Since $\N$ is $\sigma$-invariant, we do know that the cocycles
$\psi(a,b)$ and $\psi(t_{b},t_{a})$ are cohomologous, and therefore
such a function $\nu$ exists. Notice that we have some freedom in
choosing $\nu$- we can change it to be $\nu\eta$ where $\eta$ is
some character on $H$. It is easy to see by this construction that
the invertible $\kpsih$ bimodules with support in $\M$ are parameterized
by pairs $(\phi,\nu)$ where $\phi$ is a character of $H^{\nizav}$
by which it acts from the right on the module, and $\nu$ is a function
which satisfy the equation \[
\del\nu(a,b)=\psi(a,b)/\psi(t_{a},t_{b})\phi(t_{ab}t_{a}^{-1}t_{b}^{-1}).\]
 We denote the corresponding invertible bimodule by $X(\phi,\nu)$.
It is possible to choose $\psi$ and $t_{h}$ in such a way that will
assure us that $\nu|_{H^{\nizav}}$ is a character (for example- take
$\psi$ an inflation of a cocycle on $H/H^{\nizav}$ and take $t_{h}=1$
for $h\in H^{\nizav}$. We will thus assume henceforth that this is
the case.

We fix an invertible bimodule $X$ for which $\phi=1$, and for which
the restriction of $\nu$ to $H^{\nizav}$ is the trivial character
(we use here the fact that we can alter $\nu$ by a character of $H$
and the fact that any character of $H^{\nizav}$ can be extended to
a character of $H$). It is also easy to see that we can choose $\phi$
as we wish because for every choice of $\phi$, Equation \ref{equation2}
will have a solution. One last remark- notice that in that case, where
$H^{\nizav}$ acts trivially from the left and from the right, Equation
\ref{equation-emo} implies that $\nu(h)$ depends only on the coset
of h in $H^{\nizav}$. We can thus consider $\nu$ also as a function
from $H/H^{\nizav}$ to $k^{*}$.

In conclusion- we have fixed an invertible bimodule $X$ with support
in $\M$ upon which $H^{\nizav}$ acts trivially from the left and
from the right. Any other invertible bimodule with support in $\M$
will be of the form $X\otimes_{\kpsih}e$ for some $e\in E$. The
action of the nontrivial element $\sigma$ of $\Z_{2}$ on $E$ will
be conjugation by $X$, and the second obstruction is the possibility
to choose an $e\in E$ such that \[
(X\otimes e)\otimes_{\kpsih}(X\otimes e)\cong\kpsih.\]

\subsection{The action of $\sigma$ on $E$, and an explicit calculation of the second obstruction}

We would like to understand now the action of $\sigma$ on $E$. This
in turn will help us to understand the second obstruction.

As we have seen, a general element in $E$ will be a bimodule of the
form $U_{a_{i},\lambda}$. So we would like to understand what is
the bimodule $\sigma(U_{a_{i},\lambda})$.

We have the equation \[
X\otimes_{\kpsih}U_{a_{i},\lambda}=\sigma(U_{a_{i},\lambda})\otimes_{\kpsih}X.\]
 A similar calculation to the calculations we had so far reveals the
fact that if $X$ is given by $(1,\nu)$ then $U_{a_{i},\lambda}\otimes_{\kpsih}X$
is given by $(\chi(a_{i},-),\nu\lambda\chi^{-1}(a_{i},t_{-}))$, while
$X\otimes_{\kpsih}U_{a_{i},\lambda}$ is given by $(\lambda^{-1},\nu\chi^{-1}(a_{i},-)\lambda(t_{-})$.
From these two formulas we can derive an explicit formula for the
action of $\sigma$ on $E$. It follows that if $\sigma(U_{a_{i},\lambda})=U_{a_{j},\mu}$
then $j$ is the unique index which satisfies $\lambda^{-1}=\chi(a_{j},-)$
on $H^{\nizav}$, and $\mu$ is given by the formula $\mu=\chi^{-1}(a_{i},-)\lambda(t_{-})\chi(a_{j},t_{-})$.

Let us find now the second obstruction. For this, we just need to
calculate $Q:=X\otimes_{\kpsih}X$. Consider first $X\otimes X$.
It is isomorphic to $V\otimes V\otimes\bigoplus_{a\in A}(V_{a})$.
The bimodule $Q$ is the quotient of $X\otimes X$ when we divide
out the action of $k^{\psi}H$.

Let us divide out first by the action of $H^{\nizav}$. If $h\in H^{\nizav}$
we see that we divide $V\otimes V\otimes V_{a}$ by $v\otimes w-\chi(a,h)v\otimes w$.
If $a\notin H$ then there is an $h\in H^{\nizav}$ such that $\chi(a,h)\neq1$.
Therefore the support of $X\otimes_{\kpsih}X$ will be $Vec_{H}$.
Since $V$ is simple as a left and as a right $\kpsih$-module, it
is easy to see that $V\otimes_{\kpsih}V$ is one dimensional. We thus
see that $X\otimes_{\kpsih}X\cong U_{1,\lambda}$ for some character
$\lambda$. A direct calculation shows that $\lambda(h)=\nu(h)\nu(t_{h})$.
This means that the second obstruction is the character $\lambda$,
as an element of $H^{2}(\Z_{2},E)=E^{\sigma}/im(1+\sigma)$ (recall
that $\hat{H}$ is a subgroup of $E$).

Suppose that the second obstruction does vanish, and suppose that
we have a solution $X(\phi,\nu)$. In other words $X(\phi,\nu)\otimes_{\kpsih}X(\phi,\nu)\cong k^{\psi}H$.
A direct calculation similar to the one we had above shows that the
restrictions of $\phi$ and $\nu$ to $H^{\nizav}$ coincide. Recall
from Section \ref{sec:An-intrinsic-description} that if $U_{a_{i},\lambda}$
is any invertible $\kpsih$-bimodule with support in $Vec_{A}$, then
this solution is equivalent to the solution $U_{a_{i},\lambda}\otimes_{\kpsih}X(\phi,\nu)\otimes_{\kpsih}U_{a_{i},\lambda}^{-1}$.
Extend the character $\nu_{H^{\nizav}}$ to a character $\eta$ of
$H$. A direct calculation shows that $U_{1,\eta}\otimes_{\kpsih}X(\phi,\nu)\otimes_{\kpsih}U_{1,\eta}^{-1}=X(1,\nu')$.
It follows that we can assume without loss of generality that $\phi=1$.

As we have seen above, $X(\phi,\nu)^{\otimes2}\cong\kpsih$ if and
only if $\nu(h)\nu(t_{h})=1$ for every $h\in H$. So the second obstruction
vanishes if and only if there is a function $\nu$ which satisfies
equations \ref{equation2} and also the equation \begin{equation}
\nu(h)\nu(t_{h})=1\label{equation3}\end{equation}
 for every $h\in H.$ It might happen, however, that we will have
two different solutions $\nu$ and $\nu'$, that will be equivalent-
that is, there will be an invertible $\kpsih$ bimodule $U_{a_{i},\lambda}$
such that $U_{a_{i},\lambda}\otimes_{\kpsih}X(1,\nu)\otimes_{\kpsih}U_{a_{i},\lambda}^{-1}\cong X(1,\nu')$.
A careful analysis shows that this happen if and only if the following
condition holds: there is a character $\eta$ on $H$ which vanishes
on $H^{\nizav}$, such that \begin{equation}
\nu(h)/\nu'(h)=\eta(h)/\eta(t_{h}).\label{equation4}\end{equation}

In conclusion- the second obstruction is the existence of a function
$\nu:H\rightarrow H/H^{\nizav}\rightarrow k^{*}$ which satisfy Equations
\ref{equation2} and \ref{equation3}. and two such functions $\nu$
and $\nu'$ give equivalent solutions if and only if there is a character
$\eta$ of $H$ which vanishes on $H^{\nizav}$ and which satisfies
Equation \ref{equation4}.

\subsection{The third obstruction }

As explained in Section \ref{sec:An-intrinsic-description}, after
solving the second obstruction, we can think about the third obstruction
in the following way: we have an invertible $\kpsih$ bimodule $X$
with support in $\M$, and $X\otimes_{\kpsih}X\cong\kpsih$. We would
like to turn $\kpsih\oplus X$ into an algebra in $\C$. The only
obstruction for that (and this is the third obstruction) is that the
multiplication on $X\otimes X\otimes X$ might be associative only
up to a scalar. This scalar is the third obstruction, considered as
an element of $H^{3}(\Z_{2},k^{*})=\{1,-1\}$. Following the work
of Tambara (see \cite{T}), we see that this sign is the sign of the
following expression \[
\sum_{a\in H/H^{\nizav}\textrm{ }}\nu(h)\tau.\]
 If the third obstruction vanishes, we only have one possible solution,
as $H^{2}(\mathbb{Z}_{2},k^{*})=1$, since we have assumed that $k$
is algebraically closed. This finishes the proof of Proposition \ref{maintambara}

\subsection{Relation to the Tambara's Work}

In \cite{T}, Tambara classified all fiber functors on $\ty$. In
the language of module categories, he classified all module categories
over $\ty$ of rank 1. In the language of our classification, he described
all module categories whose parameterization begins with $(\M(A,\psi),\Z_{2},\ldots)$
for some $\psi$.

There is a deeper connection between our result and the result of
Tambara, as we will show now. Assume that we have a module category
over $\ty$ whose classification begins with $(\mhpsi,\Z_{2},\ldots)$.
Then, as we have seen, $H^{\nizav}<H$, and $\chi$ induces a nondegenerate
symmetric bicharacter $\bar{\chi}$ in $\bar{H}:=H/H^{\nizav}$. We
thus have another Tambara Yamagami fusion category $\D:=\mathcal{TY}(\bar{H},\bar{\chi},\bar{\tau})$,
where $\bar{\tau}$ has the same sign as $\tau$. In order to explain
the connection, we first recall the following theorem of Tambara (Proposition
3.2 in \cite{T})
\begin{thm}
Fiber functors on $\D$ correspond to triples $(s,\psi,\nu)$ which
satisfies the following coherence conditions: \[
\bar{\chi}(a,b)=\xi_{\psi}(s(a),b)\]
 \[
\del\nu(a,b)=\psi(a,b)/\psi(s(a),s(b))\]
 \[
\nu(a)\nu(s(a))=1\]
 \[
sign(\sum_{s(a)=a}\nu(a))=sign(\tau)\]
 Two such triples $(s,\psi,\nu)$ and $(s',\psi',\nu')$ will give
equivalent fiber functors if and only if $s=s'$ and there exist a
function $\phi:H/H^{\nizav}\rightarrow K$ such that $\psi=\del\phi\psi'$
and $\nu(h)/\nu'(h)=\phi(h)/\phi(s(h))$ \end{thm}
\begin{rem}
This is not exactly the original formulation in Tambara's paper, but
it is equivalent.
\end{rem}
The following lemma is now an easy corollary from Proposition \ref{maintambara}
and the above theorem.
\begin{lem}
There is a one to one correspondence between equivalence classes of
fiber functors on $\D$ which corresponds to triples which contains
the two cocycle $\psi$ and module categories over $\C$ whose parameterization
begins with $(\mhpsi,\Z_{2},\ldots)$.
\end{lem}
The Lemma says that we have a correspondence between fiber functors on
one Tambara Yamagami category and some module categories oevr another Tambara Yamagami category.
However, we do not know about a plausible
explanation of \italic{why} it happens.

We can now use the results of Tambara to obtain another description
of our module categories. Indeed, in his paper Tamabara gave several
description of fiber functors of $\D$. Applying Theorem 3.5 from
\cite{T}, we get the following
\begin{cor}
Let $\ty$, $\mhpsi$ be as above. Assume that $H^{\nizav}<H$ and
that $Rad(\psi)=H^{\nizav}$. Then the different ways to put on $\mhpsi$
a $\ty$-module structure are parameterized by pairs $(s,\mu)$ where
$s$ is an involutive automorphism of $H/H^{\nizav}$, and $\mu:\bar{H}^{s}/\bar{H}_{s}\rightarrow k^{*}$
satisfy \[
\bar{\chi}(a,b)=\xi_{\psi}(s(a),b)\]
 \[
\mu(a)\mu(b)/\mu(ab)=\tilde{\chi}(a,b)\]
 \[
sign(\mu)=sign(\tau)\]
 Here $\bar{H}^{s}$ is the subgroup of $s$-invariant elements, $\bar{H}_{s}$
is the subgroup of elements of the form $as(a)$, The map $\tilde{\chi}$
is the induced bilinear form on $\bar{H}^{s}/\bar{H}_{s}$ (one of
Tambara's result is the fact that this is indeed well defined), and
$sign(\mu)$ is the sign of $\mu$ as a quadratic map (It is quite
easy to show that $\bar{H}^{s}/\bar{H}_{s}$ is a vector space over
$\Z_{2}$ and therefore we can talk about this sign). See Tambara's
paper \cite{T} for more details.
\end{cor}

\subsection{Dual categories}

In this subsection we shall give a general description of the dual
categories of $\ty$. First recall (see \cite{ENO}) that if $\el\cong Mod_{\C}-L$
is a module category over a fusion category $\C$, where $L$ is an
algebra in $\C$, then the dual category $(\C)_{\el}^{*}$ is equivalent
as a fusion category to the category of $L$-bimodules in $\C$.

We begin with duals with respect to module categories of the form
$\el=\M(\N,1,\Phi,v,\beta)$. In this case, $\el\cong Mod_{\C}-\kpsih$
for some $H<A$ and some two cocycle $\psi$. We have described above
the category of $\kpsih$-bimodules with support in $Vec_{A}$. We
have seen that it is a pointed category with an abelian group of invertible
objects, which we have described in Subsection \ref{sub:abimods}.
Consider now the $\kpsih$-bimodules with support in $\M$. Following
previous calculations, we see that such a bimodule is given by a vector
space $V$ which is both a left and a right $\kpsih$-module, and
the interaction between the left and the right structure is given
by the formula \begin{equation}
(U_{h}\cdot v)\cdot U_{h'}=\chi(h,h')U_{h}\cdot(v\cdot U_{h'}).\end{equation}
 We can think of such modules as $k^{\theta}[H\times H]$-modules,
where $\theta$ is a suitable two cocycle. By this point of view,
the isomorphism classes of indecomposable modules is in bijection
with the characters of $Rad(\theta)<H\times H$. Let us denote the
indecomposable module which corresponds to a character $\zeta$ of
$Rad(\theta)$ by $V_{\zeta}$. A routine and tedious calculation
shows us that the group of invertible $\kpsih$-bimodules with support
in $Vec_{A}$ acts on the modules with support in $\M$ via the following
formulas: \[
U_{a_{i},\lambda}\otimes_{B}V_{\zeta}=V_{(\lambda,\chi(a_{i},-))\zeta}\]
 \[
V_{\zeta}\otimes_{B}U_{a_{i},\lambda}=V_{(\chi^{-1}(a_{i},-),\lambda^{-1})\zeta}\]
 We know that the dual category is graded by $\Z_{2}$ in the obvious
sense. We use this fact in order to conclude the following multiplication
formula: \[
V_{\zeta}\otimes_{B}V_{\eta}=\bigoplus_{(\lambda,\chi(a_i,-))t^*(\eta)=\zeta}U_{a_{i},\lambda}\]
 where by $t^{*}(\eta)$ we mean the composition of $\eta$ with the
map $H\times H\rightarrow H\times H$ given by $(h_{1},h_{2})\mapsto(h_{2},h_{1})$.
Notice that by the analysis done in Section \ref{sec:An-intrinsic-description}
and by the observation that the group of invertible bimodules with
support in $Vec_{A}$ acts transitively on the set $\{V_{\zeta}\}$,
we see that the dual is pointed if and only if the category $\el$
is $\sigma$-invariant.

We consider now module categories of the second type. By this we mean
categories of the form $\el=\M(\N,\langle\sigma\rangle,\Phi,v,\beta)$.
Assume that $\el=Mod-{Vec_{A}}\mhpsi$ over $Vec_{A}$. Then $\sigma(H,\psi)=(H,\psi)$
and we have an action of $\sigma$ on the abelian group $E$ of invertible
bimodules with support in $Vec_{A}$. We have an equivalence of fusion
categories $(Vec_{A})_{\el}^{*}\cong Vec_{E}^{\omega}$ for some three
cocycle $\omega\in H^{3}(E,k^{*})$.

We have seen in Section \ref{sec:Equivaraintization} that the dual
$(\C)_{\el}^{*}$ will be the equivariantization of this category
with respect to the action of $\Z_{2}$. If, for example, we would
have known that $\omega=1$, then this equivariantization would have
been equivalent to the representation category of the group $\Z_{2}\ltimes\hat{E}$
In general, the description of this category is not much harder.

We conclude by observing that $\ty$ is group theoretical if and only
if there is a pair $(H,\psi)$ such that $\sigma(H,\psi)=(H,\psi)$.
This gives an alternative proof of the fact that $\ty$ is group theoretical
if and only if the metric group $(A,\chi)$ has a Lagrangian subgroup (see Corollary 4.9 of \cite{GNN}).

\end{document}